\documentclass[11pt]{article}
\usepackage{graphicx}
\usepackage{amssymb,amsmath,amsthm}
\usepackage{color}
\usepackage{verbatim}
\usepackage[center, hang]{subfigure}
\newtheorem{theorem}{Theorem}[section]
\newtheorem{corollary}[theorem]{Corollary}
\newtheorem{lemma}[theorem]{Lemma}
\theoremstyle{definition}
\newtheorem{definition}[theorem]{Definition}
\theoremstyle{remark}
\newtheorem{remark}[theorem]{Remark}
\theoremstyle{definition}
\newtheorem{example}[theorem]{Example}

\numberwithin{figure}{section}

\begin{document}
\title{Homology of framed links embedded in thickened surfaces}
\author{Jeffrey Boerner\footnote{partially supported by the University of Iowa Department of Mathematics NSF VIGRE grant DMS-0602242}\\
University of Iowa\\
jboerner@math.uiowa.edu}
\renewcommand{\today}{}
\maketitle

\begin{abstract}  We construct an infinite family of homology theories of framed links in thickened surfaces, as well as a homology theory whose graded Euler characteristic is exactly the Kauffman bracket of the link in the surface.  Both theories are based on ideas coming from Asaeda, Przytycki and Sikora's categorification of the Kauffman bracket skein module of I-bundles over surfaces.  This is accomplished by borrowing ideas from Bar-Natan's Khovanov homology theory for tangles and cobordisms and using embedded surfaces to generate the chain groups, instead of diagrams.

\end{abstract}

\section {Introduction}


In \cite{K} Khovanov introduced a homology theory for links in $S^3$ that was a categorification of the Jones polynomial.  Asaeda, Przytycki and Sikora extended this theory to links embedded in I-bundles in \cite{APS} .  Their homology theory incorporated some of the topology of the I-bundle into their invariant.  


In \cite{B} the homology theory from \cite{APS} for I-bundles over orientable surfaces was constructed using embedded surfaces instead of decorated diagrams.  The theories constructed in this article modify this construction.


Section two of the paper introduces preliminary definitions and observations that pertain to all of the homology theories.  In section three the infinite family of homology theories are introduced, one for each choice of $k\in \mathbb{N} \cup \{\infty \}$.  Section four introduces the simple homology theory whose graded Euler characteristic is exactly the Kauffman bracket of the link in the surface and section five proves invariance of all of the homology theories.

\section {Definitions}

In this section we will define the elements of the chain groups for the two homology theories.  Each theory builds on these definitions slightly differently.

\begin{definition}

Let $S$ be a surface properly embedded in a 3-manifold $N$.  A boundary circle of $S$ is said to be \textbf{inessential} if it bounds a disk in $N$, otherwise it is said to be \textbf{essential}.

\end{definition}

\begin{definition}\label{comp}

If $S$ is an oriented surface and $c$ is an oriented boundary component of $S$ then the orientation of $S$ is \textbf{compatible} with the orientation of $c$ if the boundary orientation of $c$ from $S$ agrees with the orientation of $c$.    Two oriented boundary curves of an orientable connected surface are \textbf{compatible} if both curves are compatible with the same orientation on the surface.

\end{definition}

\begin{definition}
A \textbf{surface link diagram} is an orientable surface together with a link diagram in the surface.  The orientable surface will be referred to as the \textbf{underlying surface}.

\end{definition}

\begin{definition}

A \textbf{state} of a surface link diagram $D$ is a choice of smoothing at each crossing.  Therefore a state is represented by a collection of disjoint simple closed curves in the underlying surface of $D$.

\end{definition}

\begin{definition}

Let $D$ be a surface link diagram with underlying surface $F$.  A \textbf{pre-foam} with respect to $D$ has the following properties:

\begin{itemize}

\item A pre-foam is a compact surface properly embedded in $F \times I$.
\item A pre-foam has a state of $D$ as its boundary in the top ($F \times $\{0\}) and essential oriented circles as its boundary in the bottom ($F$ x \{1\}).  There are no other boundary components.

\item Pre-foams may be marked with dots.

\end{itemize}

Two pre-foams are equivalent if they are isotopic relative to the boundary.  The dots on pre-foams are allowed to move freely within components but dots may not switch components.  

\end{definition}

\begin{definition}Let $M_D$ be the free $\mathbb{Z}$-module generated by pre-foams with respect to the surface link diagram $D$.\end{definition}

\medskip

\begin{definition} Let \textbf{$B$} be the submodule of $M_D$ generated by the following relations: 
\end{definition}

\begin{enumerate}

\item The neck-cutting relation. (NC)
\begin{center}
\includegraphics[height=1 in]{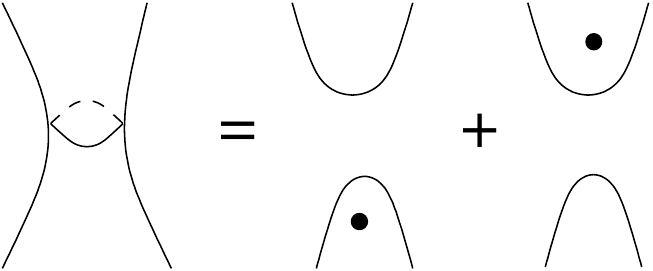}
\end{center}

\item A sphere bounding a ball equals zero. (SB)
\begin{center}
\includegraphics[height=1 in]{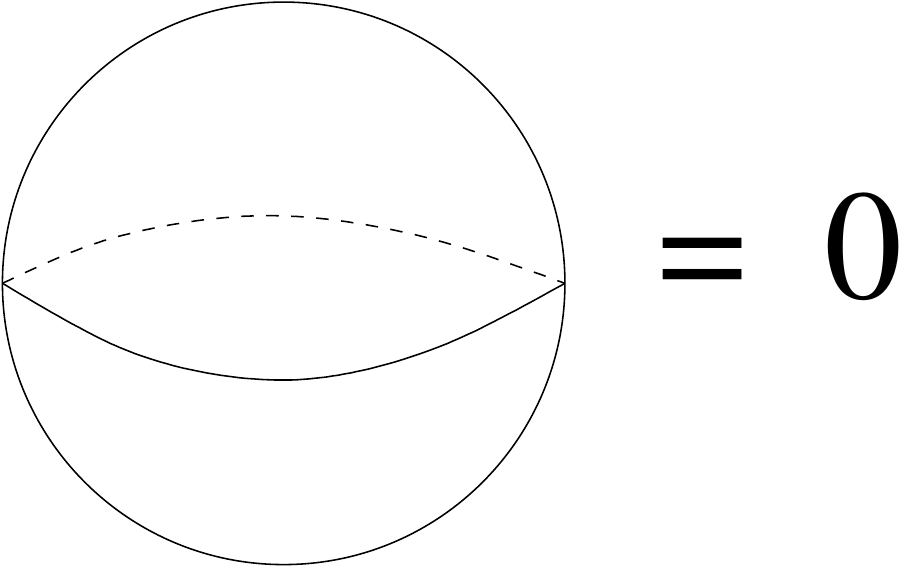}
\end{center}

\item A sphere with a dot bounding a ball equals one. (SD)
\begin{center}
\includegraphics[height=1 in]{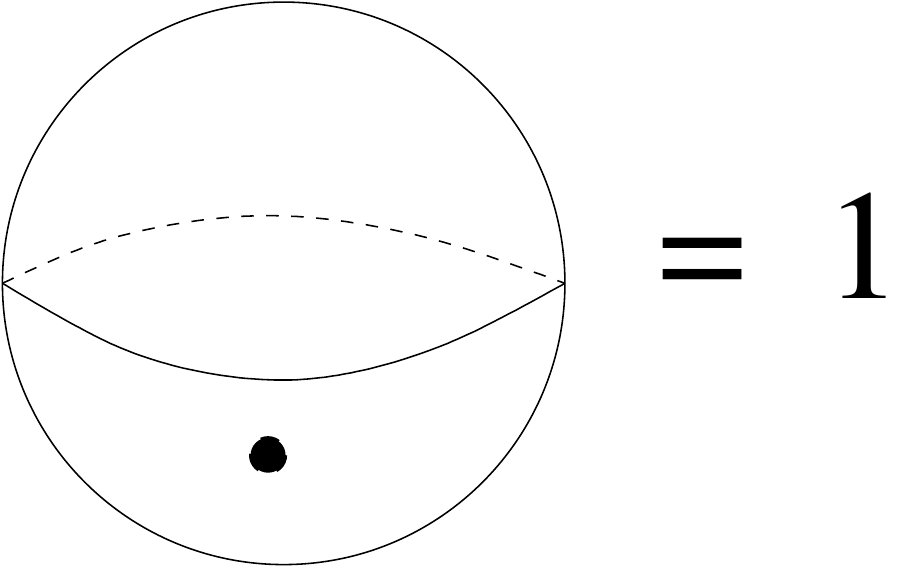}
\end{center}

\item A component with two dots equals zero. (TD)
\begin{center}
\includegraphics[height=.6 in]{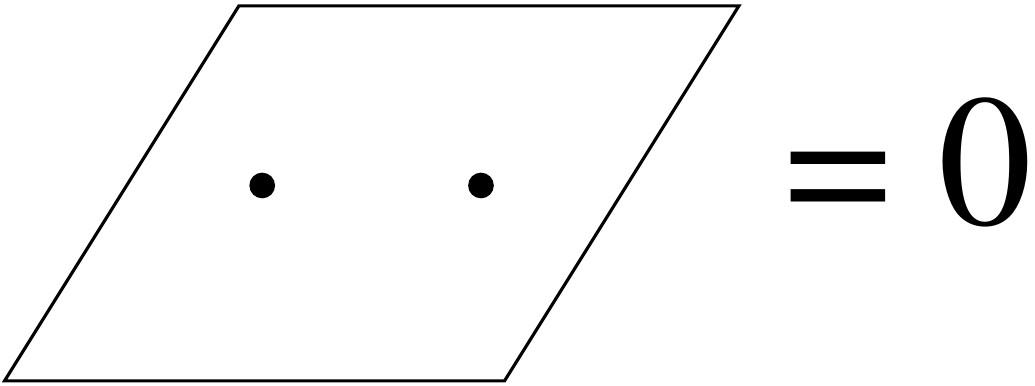}
\end{center}

\item A surface with a non-disk, non-sphere incompressible component, that has a dot on that component equals zero.(NDD)

\end{enumerate}

\begin{definition}
Elements of $M_D/B$ are called \textbf{foams} with respect to the diagram $D$.
\end{definition}

\subsection{Observations about foams}

We will now make some observations about foams that will be useful later. 

\begin{lemma}
\label{nontrivdot}
A foam with a dotted component that has an essential boundary component is trivial in the quotient.

\end{lemma}

\begin{proof}
If the dotted component is incompressible we are done by the NDD relation.  If not, then apply the NC relation to this component.  If the compressing disk was non-separating, then the result of compressing is a component with two dots, thus it is trivial in the quotient by the TD relation.

If the compressing disk is separating the result of the NC relation is two components. One of them has an essential boundary curve since the NC relation does not affect boundary components.  Therefore one of the components is not a disk, but they both have dots.

We are now left with two components, each with dots, and one of them has an essential boundary component. 

Note the component that has an essential boundary curve and a dot satisfies the assumptions of the Lemma.  Thus we may continue in this manner.  Note that each time the NC relation is applied the Euler characteristic goes up by two. Since we are dealing with compact surfaces the Euler characteristic is bounded above,  thus the NC relation can only be applied a finite number of times.  When the NC relation can no longer be applied there is an incompressible component with a dot that is not a disk or a sphere.

\begin{center}
\includegraphics[height = 2 in]{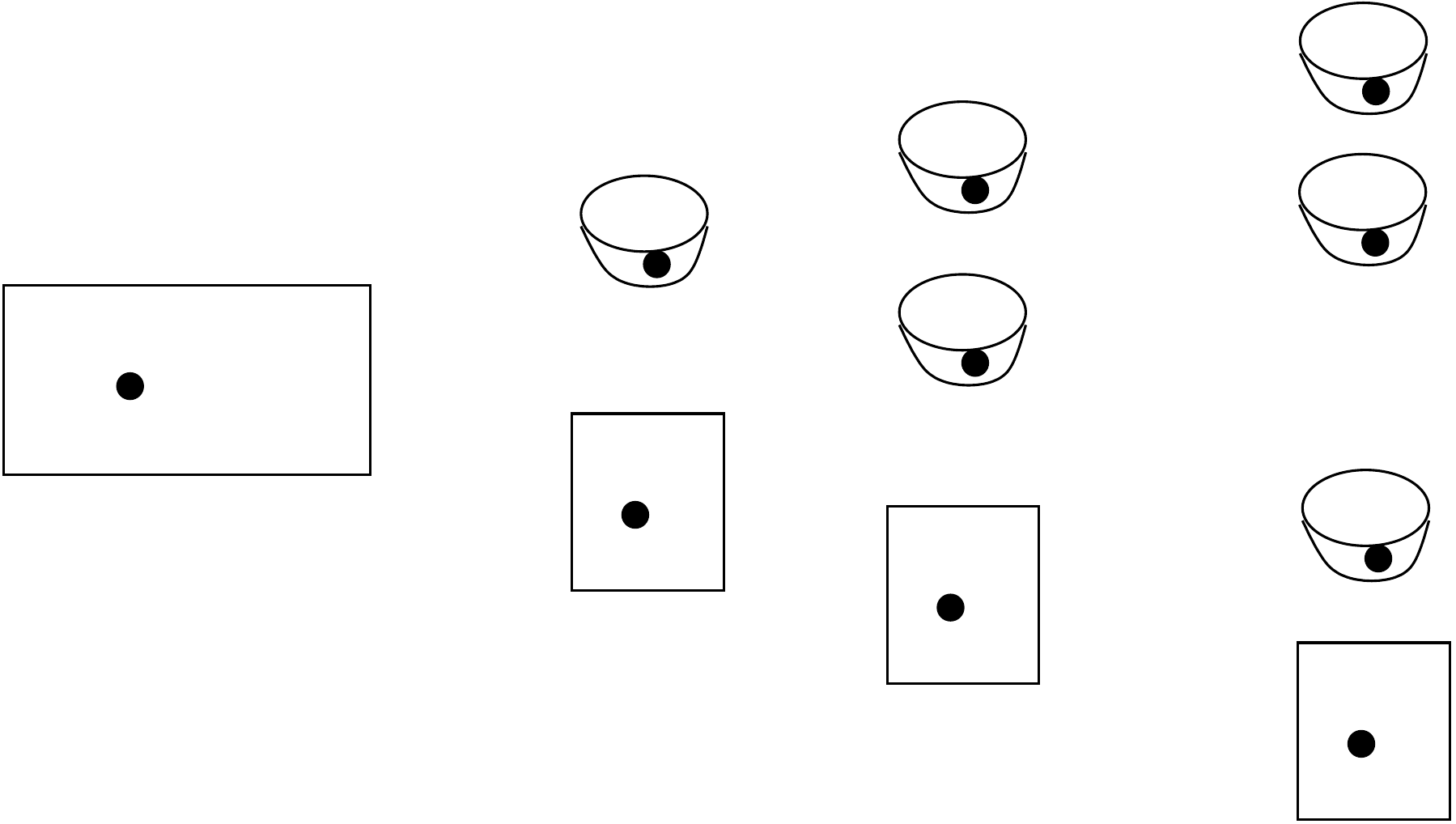}\put(-180,80){$=$} \put(-115,80){$=$} \put(-70,80){$=\dots=$}\put(-15,80){$\vdots$}\put(-240,10){Incompressible, non-disk and non-sphere $\rightarrow$}     
\end{center}

It can be seen in the figure that the original surface is equal to a foam that has an incompressible non-disk and non-sphere component with a dot, which is trivial in the quotient by the NDD relation.

\end{proof}

\begin{lemma}
\label{closedcomp}

Foams with closed dotted non-sphere components are trivial.

\end{lemma}

\begin{proof}

If the closed component is incompressible and it has a dot then is trivial in the quotient, by the NDD relation.  

If the closed component is compressible, then the NC relation can be applied.  If the compressing disk is non-separating then after the NC relation is applied the new closed component has two dots and is thus trivial in the quotient.

If the compressing disk is separating then we are left with a sum of closed non-sphere surfaces, one of which has a dot in each summand.  Note that each summand satisfies the assumptions of this Lemma.  Since we are dealing with compact foams and the Euler characteristic is bounded above the process will eventually terminate, as in Lemma \ref{nontrivdot} .

\end{proof}

\begin{corollary}
\label{esscompr}
If a foam has a non-torus, non-sphere component that is compressible with no inessential boundary curves then it is trivial in the quotient.
\end{corollary}

\begin{proof}

Since the component is compressible the NC relation can be applied.  After the NC relation is applied the only surfaces present are closed surfaces of non-zero genus, or surfaces with essential boundary curves.  One of these has a dot in each summand and thus the surface is trivial in the quotient by Lemma \ref{nontrivdot}.

\end{proof}

\section{$k$-Homology}

The first homology theory presented is actually an infinite family of homology theories.  There is one theory for each choice of $k\in \mathbb{N} \cup \{\infty\}$.  We must define conditions on orientations of the boundary components of foams and additional relations for these theories.

\begin{definition}
A foam is \textbf{$F$-oriented} if it satisfies two conditions:

\begin{enumerate}

\item Inessential boundary curves of state surfaces are not oriented, but essential boundary curves may or may not be oriented.  However, essential curves in the bottom must be oriented.

\item If one component of a state surface has an oriented essential boundary curve, then all essential boundary curves on that component must be oriented compatibly with respect to Definition \ref{comp}.

\end{enumerate}

\end{definition}

\begin{definition}

$B_k$ is a submodule of $M/B$ generated by the relations:

\begin{enumerate}

\item If $k \neq \infty$ then an incompressible component with Euler characteristic less than $-k$ is zero. (KEC)

\item An incompressible non-orientable component is zero.  (NOS)

\end{enumerate}

\end{definition}

\begin{definition}
\textbf{$k$-foams} with respect to the surface link diagram $D$ are $F$-oriented elements of $M_D/(B \cup B_k)$.
\end{definition}

\subsection{Chain Groups and Grading}

\begin{definition}

Let $S$ be a $k$-foam.

$p(S) = \#$ of positive smoothings in the state corresponding to the top boundary of $S$,

$n(S) = \#$ of negative smoothings in the state corresponding to the top boundary of $S$,

$I(S) = p(S) - n(S)$,

$J(S) = I(S) + 2(2d  - \chi(S))$  where $d$ is the number of dots on $S$ and $\chi(S)$ is the Euler characteristic of $S$.

\end{definition}

\medskip

The collection of $k$-foams are tri-graded. The first two gradings are $I$ and $J$ and the third index corresponds to the oriented essential disjoint simple closed curves in the bottom of $F \times I$.  Note that given two parallel oriented simple closed curves on a surface it is not difficult to determine if their orientations agree or disagree.  To define the third index it is necessary to specify that one of the orientations possible is the positive one and the opposite orientation is the negative one.  Thus for each homotopy class of simple closed curve we assume we have chosen a positive orientation and a negative orientation.

\begin{definition}
Let $\gamma_1,\dots,\gamma_n$ be a family of disjoint simple closed curves in the bottom of a $k$-foam $S$.    If $\gamma_i$ and $\gamma_j$ are parallel then $\gamma_i=\gamma_j$.

Then 

$K(S) = \sum_{i=1}^n k_i\gamma_i$, where 

$$   
k_i  = \left\{ 
\begin{array}{ccc}
1,&\mbox{ if } \text{$\gamma_i$ is oriented in the positive direction}                  \\
-1,&\mbox{ if } \text{$\gamma_i$ is oriented in the negative direction}                  \\
\end{array}\right.
$$

\end{definition}

\medskip

In order for these definitions to make sense they need to be well defined and respect the relations on $k$-foams.  Thus we have the following Lemma:

\begin{lemma}
The indices $I(S)$, $J(S)$ and $K(S)$ are well defined for a k-foam $S$.
\end{lemma}

\begin{proof}
Showing the indices $I(S)$ and $K(S)$ are well-defined is immediate since the circles in the top and bottom are not affected by the relations on $k$-foams.

In order to show the index $J(S)$ is well defined we need to consider the NC relation and the SD relation.

To consider the NC relation, first note that when NC is applied the Euler characteristic goes up by two.  Also, each summand adds a dot, thus $2d  - \chi(S)$ remains the same.  

A sphere has Euler characteristic two, and note if the sphere has a dot, then $2d-\chi(S) = 2 -2 =0$, so removing a sphere with a dot does not affect $J(S)$.

\end{proof}

\medskip

\begin{definition}Let $C_{i,j,s}(D)$ be the free module generated by all $k$-foams with respect to the diagram $D$, $S$, such that  $I(S)=i$, $J(S)=j$ and $K(S)=s$.
\end{definition}

\subsection {The Boundary Operator}

Let $p$ be a crossing of the link diagram $D$.  

\begin{definition}

Let $S$ be a $k$-foam.  \textbf{Placing a bridge at the $p^{\text{th}}$ crossing} of $S$ is depicted in Figure \ref{placebridge}.

\begin{figure}[h!]

\begin{center}
\begin{tabular}{ccc}
 \includegraphics[height=1 in]{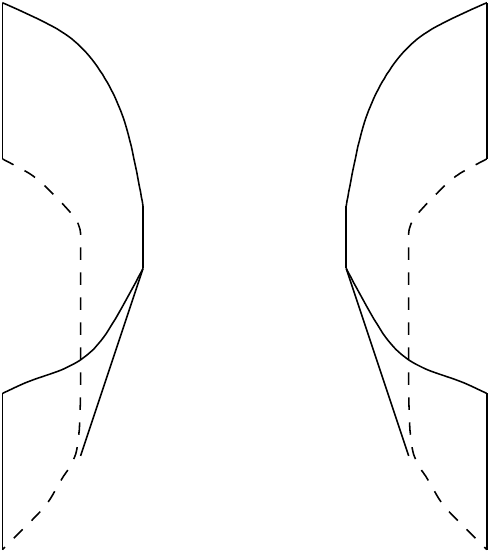} & \hspace{.2 in} \raisebox{.4 in}{$\mapsto$} \hspace{.1 in} &\includegraphics[height=1 in]{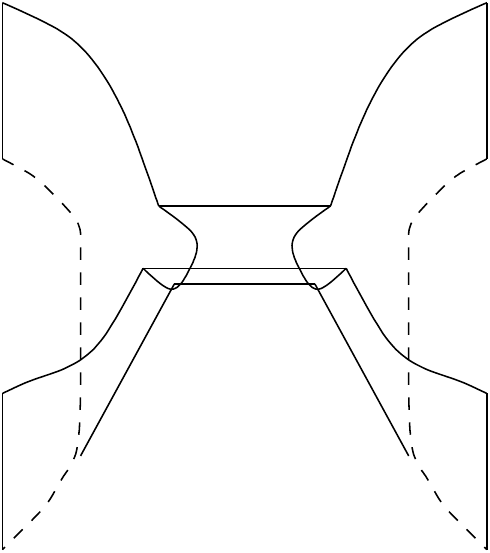} \\

 $S$ at $p^{\text{th}}$ crossing & &   After bridge is placed
\end{tabular}
\caption{Placing a bridge at the $p^\text{th}$ crossing     \label{placebridge}}
\end{center}
\end{figure}

\end{definition}

\bigskip

\begin{remark}
If a bridge is placed between two subsurfaces with oriented curves, the orientations may conflict.  Since all curves on a component must be oriented compatibly we can consider each component to be oriented in the way compatible with its boundary curves.  If no boundary curves are oriented, the component can be considered unoriented.
\end{remark}

\begin{definition}

When a bridge is placed at the $p^{\text{th}}$ crossing of a $k$-foam $S$, the orientation from one side of the bridge can be slid along the bridge to the other side of the bridge.  If the orientation that is slid across the bridge disagrees with existing orientation, then \textbf{EO} is said to occur at the $p^{\text{th}}$ crossing at $S$.

\begin{figure}[h!]
\label{EOpic}
\begin{center}
\includegraphics[height =1.5in]{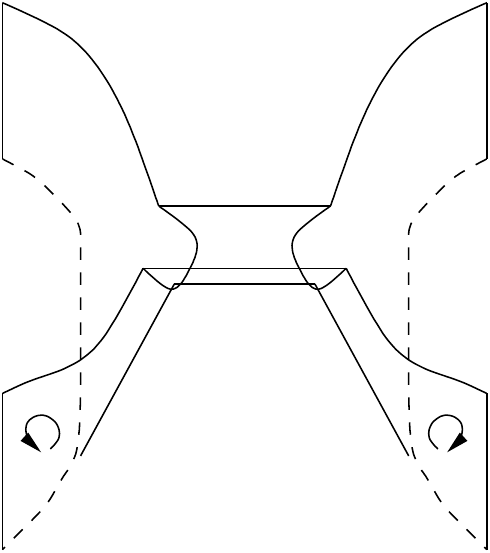}
\caption{The is an example of when EO occurs.}
\end{center}
\end{figure}
\end{definition}

\begin{definition}

If EO does not occur at the $p^{\text{th}}$ crossing, then the components must be oriented compatibly.  \textbf{$d'_p(S)$} is $S$ but with a bridge placed at the $p^{\text{th}}$ crossing.   In addition, $d'_p(S)$ is oriented so that all essential curves are oriented compatibly with the existing orientation on the components and inessential components are unoriented.

\end{definition}

Now we are able to define the boundary operator of the chain complex.

\begin{definition}
$$
d_p(S) = \left\{ 
\begin{array}{ccc}
0,&\mbox{ if } \text{$p^{\text{th}}$ crossing of $S$ is smoothed negatively}                  \\
0,&\mbox{ if } \text{(EO) occurs}                  \\
\text{$d'_p(S)$}, &\mbox{ else }
\end{array}\right.
$$

Then we define the boundary operator $d: C_{i,j,s}(D) \rightarrow C_{i-2,j,s}(D)$, by 

\begin{equation}d(S) = 
\sum_{p \text{ a crossing of $D$}} (-1)^{t(S,p)}d_p(S),\end{equation}
where $t(S,p) = |$\{ $j$ a crossing of $D$ : $j$ is after $p$ in the ordering of crossings and $j$ is smoothed negatively in boundary state of S\}$|$

\end{definition}

\begin{example}

It may seem that the order of applying $d_p$ and applying the NC relation may affect the orientation on the boundary curves.  Consider the $k$-foam in Figure \ref{orientexamp}.  If a bridge is placed at the dotted line, orientation is forced on any new essential curves.  However, if the NC relation was applied before placing the bridge no orientation would be forced.  Notice after applying the NC relation one component has a dot in each summand and they have essential curves, so they are trivial $k$-foams.  Lemma \ref{ncwelldef} shows that this is what happens in general.

\begin{figure}[h!]
\begin{center}
\includegraphics[height=1.5 in]{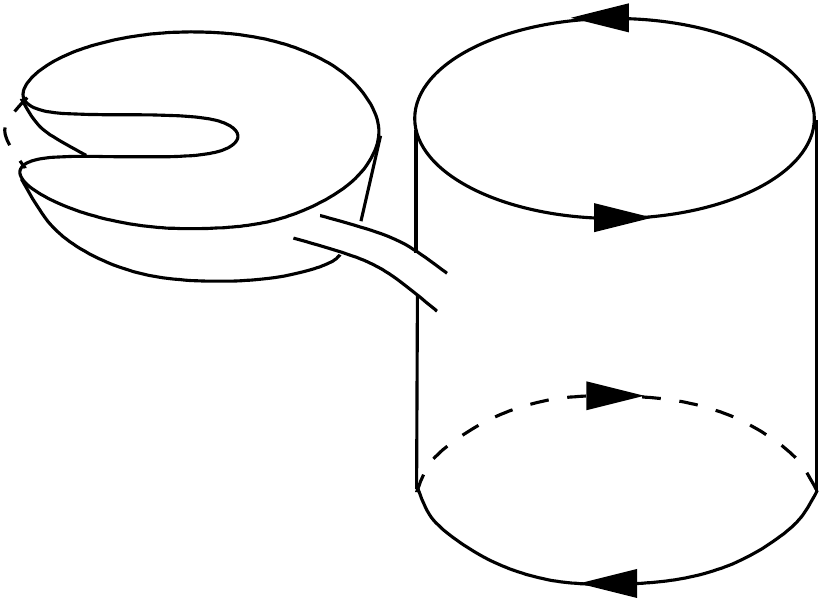}
\caption{A $k$-foam where the order of applying the neck-cutting and applying the boundary operator affects the orientation. \label{orientexamp}}
\end{center}
\end{figure}

\end{example}

\begin{lemma}
\label{ncwelldef}

The boundary operator together with the NC relation is well-defined with regard to orientations of $k$-foams, i.e. orientations of curves agree after applying the boundary operator regardless of when the NC relation is applied.

\end{lemma}
\begin{proof}

Note that on the left side of the NC relation boundary orientation from essential boundary circles are forced on other essential boundary circles through the neck.  However, on the right side of the relation all essential boundary curves do not need to be compatible with one another since they may no longer lie on the same connected component.

\begin{center}
\includegraphics[height = .8 in]{BNrel.pdf}
\end{center}

Orientation can only be forced in the left side of the equation if each component has an essential boundary component on the right-hand side.  Note one component has a dot in each summand, so by Lemma \ref{nontrivdot} both sides are trivial in the quotient.  Thus orientations may differ, but the foams are trivial, and therefore equal in the quotient.  

\end{proof}

\begin{remark}

Here are two observations that are helpful in the following proofs:

\begin{enumerate}

\item Bridging an essential boundary curve to itself can produce at most one inessential boundary curve at a time.  This is due to the fact that if two inessential boundary curves are created from one curve, then the original curve was also inessential.

\item If a component has an inessential boundary component then either this component is a disk or it is compressible. (Just push the disk the curve bounds into $F \times I$ to obtain a compressing disk.)

\end{enumerate}

\end{remark}

\begin{theorem}
\label{bopwelldef}
The boundary operator, $d$, is well defined on $k$-foams.
\end{theorem}
\begin{proof}
Note the boundary operator is defined on the original module but it actually operates on a quotient of that module.  Thus it needs to be verified that two representations of the same class go to the same class under the boundary operator.  Thus assume $[S] = [S']$ in our quotient.  Therefore $S-S' \in B\cup B_k$.  So we must show that $d(S) - d(S') \in B\cup B_k$.  Since $d$ is linear this is equivalent to showing $d(S-S') \in B\cup B_k$.  Therefore it is sufficient to show that given $b \in B\cup B_k$, that $d(b)\in B\cup B_k$.

Therefore we need to look at the relations that define $k$-foams and see how the boundary operator affects them.  Note that spheres bounding balls and surfaces with two dots are unaffected by the boundary operator, so we do not need to consider the SB, SD, or TD relations.  This leaves the NC, NDD, KEC, and NOS relations to check.  We may assume that the bridge is being placed on the component that has one of the conditions, otherwise the result is immediate.  Also note that if the component in question is bridged to a disk, then the result is again immediate, since adding a disk does not change the component.  Also, since we can assume the surfaces are incompressible by the NC relation we are never bridging to an inessential curve.  We will address these case by case.

\begin{enumerate}

\item[(NC)] 

It must be shown that if foams are related by the neck-cutting relation before applying the boundary operator they are related after applying the boundary operator as well.  Placing a bridge does not remove any compressing disks, so the only issue is how orientations are affected.  Lemma \ref{ncwelldef} shows if the orientations agree before applying the boundary operator, they agree after applying it as well.

\item[(KEC)] 

Assume the boundary operator is applied to a $k$-foam that has an incompressible component with Euler characteristic less than or equal to $-k$.  When the bridge is placed the new component has Euler characteristic less than or equal to $-k-1$.  At most one inessential boundary curve can be created when a bridge is placed.  

If no inessential curves are created, then all curves are essential.  If the component is incompressible, we are done by KEC.  If the component is compressible then we are done by Corollary \ref{esscompr}.

If an inessential curve is created, then apply the NC relation to yield a disk with a dot and a surface of Euler characteristic less than or equal to $-k$.

If the non-disk surface is incompressible, the $k$-foam is trivial.  If the non-disk is compressible then the foam is trivial by Corollary \ref{esscompr}.

\item[(NDD)] 

Assume the boundary operator is applied to an incompressible non-disk dotted component (thus this component has Euler characteristic less than or equal to zero).  After placing the bridge the new connected component has Euler characteristic less than or equal to negative one and it has a dot. If this new component has an essential boundary component then it is trivial in the quotient by Lemma \ref{nontrivdot}.  

Otherwise this component has no essential boundary components and so all boundary components must be inessential.  Since it was originally incompressible it must have only one inessential boundary component after placing a bridge.  There is a compressing disk present since this component is not a disk, but has an inessential boundary curve.  This compressing disk is separating and compressing upon it yields a disk and a closed surface with a dot.  This closed surface is not a sphere since there was a compressing disk.  The entire foam is now trivial in the quotient by Lemma \ref{closedcomp}.

\item[(NOS)]

Assume the boundary operator is applied to a $k$-foam that has an incompressible non-orientable component.  After placing the bridge the component is still non-orientable because a bridge will not change the fact that the component is one-sided.   If inessential curves are created, the neck-cutting relation can be applied to separate them from the rest of the component.  Notice that this operation still leaves the component non-orientable.  If the non-orientable component is compressible the the entire foam is zero by Corollary \ref{esscompr}, since there are no remaining inessential curves.  If the component is now incompressible, then the foam is zero by the NOS relation.

\end{enumerate}

When the boundary operator is applied to elements of $B\cup B_k$, they remain in $B\cup B_k$, so by the remarks at the beginning of this proof, the boundary operator is well defined on the quotient.

\end{proof}

\subsection{$d^2 = 0$}

In order to show $d^2=0$, we need to show that $d_p\circ d_q = d_q\circ d_p$ $\forall p,q$ crossings.  The following lemmas accomplish that.

\begin{definition}

Two boundary curves are said to be \textbf{connected} if they are connected by a former crossing smoothed positively.

\begin{figure}[htb]
\begin{center}
  \includegraphics[height = 1 in]{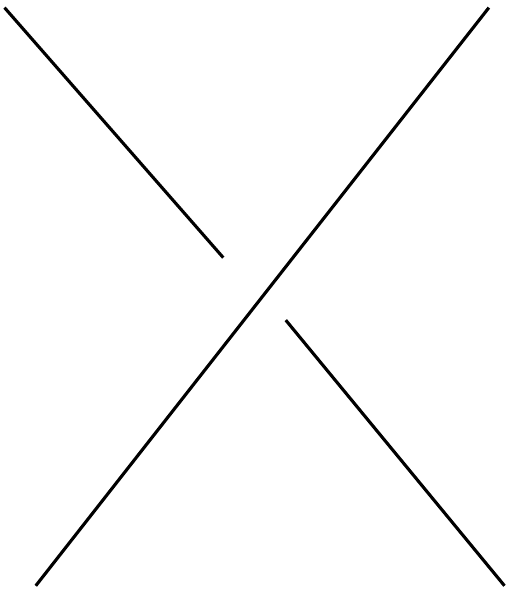} \raisebox{.5 in}{$\rightarrow$} \includegraphics[height = 1 in]{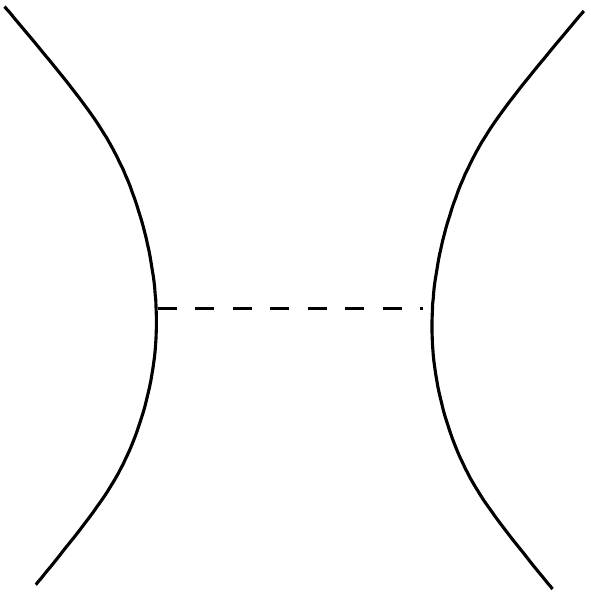}  
\end{center}
\caption{The two curves on the right are said to be connected.}
\end{figure}

\end{definition}

\begin{lemma}

Let $a$ and $b$ be two crossings of a given diagram and let $S$ be a $k$-foam.  If $d_a(d_b(S)) = 0$ because of EO, then $d_b(d_a(S))$ is a trivial $k$-foam.
\end{lemma}

\begin{proof}
\label{eocomm}

Assume $d_a(d_b(S)) = 0$ because of EO.  This means at least two oriented boundary curves must be present in the foam and these oriented curves can be, at most, two components away from one another (i.e. both connected to the same boundary component.)

First assume the oriented curves are connected.  Then assume EO occurs at $b$.  If a bridge is placed at $a$ then the curves connected by $b$ can either continue to have conflicting orientation, or one of them may become a curve on a component with a dot (after applying NC).  This is due to the fact that by the orientation rules of $\bar{d}_p$, oriented curves can only lose their orientation, not reverse their orientation, and losing orientation only happens when a curve becomes inessential.  In either case placing a bridge at $b$ after $a$ results in something trivial in the quotient.

Now assume the oriented curves are both connected to an intermediate curve.  Assume without loss of generality this curve is unoriented as in Figure \ref{EOproof}.  (If it is oriented, it cannot be compatible with both curves, since then EO would not occur, so we can reduce to the previous case.)  Placing a bridge at either of the crossings orients the previously unoriented curve in such a way that it is not compatible with the other curve.  Thus when the second bridge is placed in either order there is EO, so $d_b(d_a(S))$ is a trivial $k$-foam.

\begin{figure}[htb]

\begin{center}
\includegraphics[scale = .5]{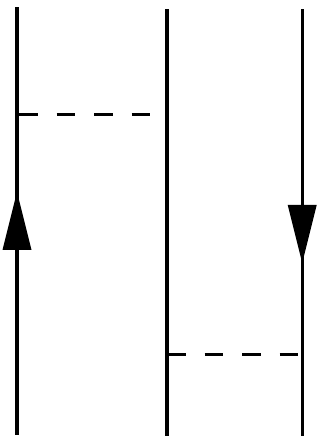}
\end{center}
\caption{Two oriented curves connected to the same unoriented curve. \label{EOproof}}
\end{figure}

\end{proof}

\begin{lemma}
\label{oragree}

Let $S$ be a foam and $a$ and $b$ are crossings in the associated diagram.  Then ${d}_a({d}_b(S))$ has the same orientation on boundary curves as $d_b(d_a(S))$. 

\end{lemma}
\begin{proof}

The lemma is immediate if the affected components have more than two essential curves.  This is due to the fact that at most two essential curves can become inessential when a bridge is placed.  If all essential curves are eliminated, then no curves are oriented, so the orientations necessarily agree.

We can then conclude that one of the orderings of crossings has no essential curves on the affected components after one crossing is changed.  Without loss of generality assume $d_a(S)$ has no essential curves.  Thus $d_a(S)$ has one boundary curve and it is inessential.  Applying the NC relation to a curve parallel to the inessential curve yields a disk and a closed surface in a sum with the dots distributed appropriately.  Note the closed surface with a dot is not a sphere and thus is trivial in the quotient by Lemma \ref{closedcomp}.   Thus the only surface left is a closed surface and a disk with a dot.  The boundary curve on the disk cannot become essential, since then the foam would be trivial as it has a dot, by Lemma \ref{nontrivdot}. Since only essential curves can be oriented, the boundary curve cannot become oriented either.  Therefore $d_b(d_a(S))$ is either trivial or unoriented, and in either case the orientation agrees with that of $d_a(d_b(S))$.

\end{proof}

\begin{theorem}
$d^2=0$

\end{theorem}

\begin{proof}
Note that by the definition of the boundary operator all that needs to be shown is that given two crossings $a$ and $b$ and a foam $S$, that $d_a(d_b(S)) = d_b(d_a(S))$.  This is clearly the case if EO does not occur and if we disregard orientation.

If EO does occur we have shown by Lemma \ref{eocomm} that if it occurs for one order of a and b, then  $d_a(d_b(S))$ and $d_b(d_a(S))$ are both trivial in the quotient.  By Lemma \ref{oragree}, $d_a(d_b(S))$ and $d_b(d_a(S))$ are oriented identically.

Thus the partials always commute, so after adding in the appropriate negative signs, $d^2=0$.

\end{proof}

\section{A simpler theory}

We will describe a simpler method to obtain a homology theory that determines the Kauffman bracket than the one from \cite{APS}.  It is simper in the sense that the graded Euler characteristic of the chain complex is exactly the Kauffman bracket.

\begin{definition} Let \textbf{$B_{S}$} be the submodule of $M_D/B$ generated by the following relations: 
\end{definition}

\begin{enumerate}

\item An incompressible annulus with both boundary components in the top or bottom equals zero. (TT)

\item A surface with an incompressible component with negative Euler characteristic equals zero.   (NEC)

\item An incompressible torus is zero. (IT)

\end{enumerate}

\begin{definition}
The elements of $M_D/(B\cup B_{S})$ are called \textbf{simple foams} with respect to the link diagram $D$.
\end{definition}

We now recall a result from Friedhelm Waldhausen's paper, \emph{On Irreducible 3-manifolds Which are Sufficiently Large}:

\begin{theorem}[Waldhausen]\label{waldhausen}

Let $F$ be an orientable surface.  Every incompressible surface with its boundary completely in the top of $F \times I$, is parallel to a subsurface of $F$.

\end{theorem}

\begin{corollary}\label{mobius}

Non-trivial simple foams cannot have incompressible non-orientable components.

\end{corollary}

\begin{proof}
If a non-orientable component is incompressible and has zero or one boundary component, then by Theorem \ref{waldhausen} such a component would be parallel to the top of the thickened surface.  However, the top of the thickened surface is orientable, so this is not possible.  All non-orientable components with two boundary components have negative euler characteristic, thus it is not possible to have a non-trivial simple foam with an incompressible non-orientable component.
\end{proof}

\subsection{Boundary operator}

We will now define a chain complex on simple foams.  We will use the $I$ and $J$ indices as defined for $k$-homology.  The $\bar{K}$ index is similar to the $K$ index, however essential curves are no longer oriented, so they don't have a sign associated to them.  

\begin{definition}

Let $S$ be a simple foam.  \textbf{$\bar{K}(S)$} is the diagram of the essential curves in the bottom of the foam.

\end{definition}

\begin{example}  If $S$ is the simple foam in figure \ref{thindexamp} then $\bar{K}(S) =$ \includegraphics[scale = .5]{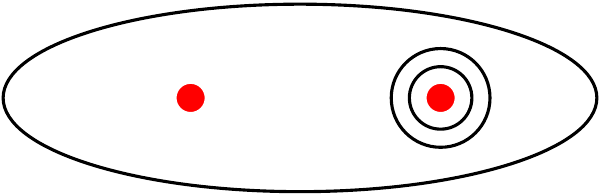}.

\begin{figure}[h!]
\begin{center}
\includegraphics[height = 1 in]{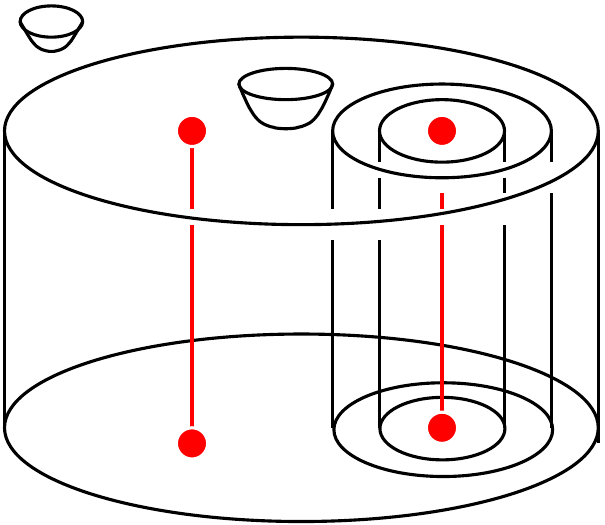}
\end{center}
\caption{The simple foam is in the thickened twice punctured disk.  The vertical lines represent the punctures. \label{thindexamp}}
\end{figure}

\end{example}

\begin{definition}Let $\bar{C}_{i,j,s}(D)$ be the chain group generated by simple foams with respect to the diagram $D$, that can be represented by incompressible vertical annuli and disks with the appropriate $I$, $J$ and $\bar{K}$ grading.
\end{definition}

\begin{definition}
Let $S$ be a simple foam, then $\hat{d}_p(S)$ is $S$, but with a bridge placed at the $p^{\text{th}}$ crossing.
\end{definition}

\begin{lemma}
$\hat{d}_p: \bar{C}_{i,j,s}(D) \rightarrow \bar{C}_{i-2,j,s}$.
\end{lemma}
\begin{proof}

If $S \in C_{i,j,s}$, the fact that $I(\hat{d}_p(S)) = i-2$ and $J(\hat{d}_p(S)) = j$ is immediate by Euler characteristic and the fact that changing a positive smoothing to a negative smoothing reduces $I$ by 2.  The curves in the bottom are not affected by the boundary operator, so the $\bar{K}$ index stays constant as well.

The only additional item to show is that $\hat{d}_p(S)$ can be represented in terms of foams that consist only of incompressible vertical annuli and disks since these generate the chain groups.

By the relations every foam can be represented by incompressible components each having non-negative Euler characteristic.  Non-orientable surfaces are again impossible by Corollary \ref{mobius}.  This leaves disks, incompressible annuli and incompressible tori.  Incompressible tori are zero by the IT relation.

\end{proof}

\begin{definition}

If two distinct essential curves are connected by the crossing $p$, then \textbf{EC} is said to occur at the $p^\text{th}$ crossing.

\end{definition}

\begin{definition}$\bar{d}_p:\bar{C}_{i,j,s}(D) \rightarrow \bar{C}_{i-2,j,s}(D)$

$$
\bar{d}_p(S) = \left\{ 
\begin{array}{ccc}
0,&\mbox{ if } \text{$p$-th crossing of $S$ is smoothed negatively}                  \\
0,&\mbox{ if } \text{EC occurs at the $p$-th crossing of $S$}                  \\
\text{$\hat{d}_p(S)$}, &\mbox{ else }
\end{array}\right.
$$
\end{definition}

Then $\bar{d}: \bar{C}_{i,j,s}(D) \rightarrow \bar{C}_{i-2,j,s}(D)$ is defined by

\begin{equation}\bar{d}(S) = 
\sum_{p \text{ a crossing of $D$}} (-1)^{t(S,p)}d_p(S),\end{equation}
where $t(S,p) = |$\{ $j$ a crossing of $D$ : $j$ is after $p$ in the ordering of crossings and $j$ is smoothed negatively in boundary state of S\}$|$

\begin{theorem}

The boundary operator, $\bar{d}$, is well defined with respect to simple foams.

\end{theorem}

\begin{proof}

We must show that trivial simple foams remain trivial under the boundary operator.  This proof closely follows the proof of Theorem \ref{bopwelldef}.  In that proof the NC and NDD relation were addressed.  This only leaves the TT, NEC and IT relations.  We will address these by cases.

\begin{itemize}

\item[(NEC)]

Assume the boundary operator is applied to a simple foam with an incompressible negative Euler characteristic component.  An incompressible negative Euler characteristic component has only essential boundary curves.  EC occurs unless the same essential curve bridges to itself.  If the component surface is incompressible the foam is trivial by the NEC relation.  If the surface is compressible with no inessential curves, then the foam is trivial by Lemma \ref{esscompr}.

If an inessential curve is created then we can remove it from the rest of the component by the NC relation.  This leaves a disk and a surface of negative Euler characteristic.  If the non-disk component is compressible the foam is trivial by Lemma \ref{esscompr}.  If the non-disk component is incompressible, then the foam is trivial by the NEC relation.

\item[(TT)] 

Assume the boundary operator is applied to a simple foam with an incompressible top annulus.  As before we can assume that the bridge placed on a top annulus bridges the same boundary curve to itself, since if it bridged to another boundary curve, the result would be zero by EC.

If the bridge connects a curve to itself, then there is now either a compressing disk or the surface is incompressible with negative Euler characteristic. The incompressible surface is trivial in the quotient, so we may assume it is compressible.  The surface has either one inessential curve or no inessential curves.  By Corollary \ref{esscompr} we can assume one of the boundary components is inessential.  Thus after compression there is a disk and the the rest of the surface.  By noting the Euler characteristic what is left is a top annulus which is trivial in the quotient.

\item[(IT)]

If a simple foam has an incompressible torus component, then after applying the boundary operator it will continue to have an incompressible torus component, so it remains trivial in the quiotient.

\end{itemize}

\end{proof}

\begin{lemma}
If $d_a(d_b(S)) = 0$ by EC, then $d_b(d_a(S))$ is trivial in the quotient.
\end{lemma}
\begin{proof}

If $d_a(d_b(S)) = 0$ by EC then there must be at least two essential boundary curves in the top of the foam and they must either be connected, or both are connected to the same curve.  

If they are both connected to the same curve we may assume the intermediate curve is inessential otherwise we are in the other case.  When an essential and an inessential curve combine we are left with an essential curve, thus in either order, the result is zero.

If two essential curves are connected to each other we can assume one crossing connects them, while another crossing connects one of the curves to itself.  Otherwise, if it connects anywhere else the other crossing still connects two essential curves or an essential and inessential.  Either of which results in an essential curve.  When a bridge is placed at the self-connecting crossing an inessential curve may be produced along with an essential curve.  After applying the NC relation the other crossing either connects an essential curve to an essential or an essential to an inessential on a dotted disk.  In either case the result is trivial in the quotient.

\end{proof}

\begin{theorem}

$d^2=0$

\end{theorem}
\begin{proof}

The partial boundary operators commute if EC does not occur since the resulting foams are isotopic.  The preceding lemma shows that if EC occurs, the partials commute then too.  Thus with the appropriate negative signs, the partial boundary operators anti-commute, and since $d$ is a sum of the partial boundary operators with the negative signs, we have $d^2=0$.

\end{proof}

\subsection{Graded Euler characteristic}

\begin{definition}
The graded Euler characteristic of a chain complex is

$$\chi(C_{*,*,s}(D)) = \sum_{i,j} A^j(-1)^{\frac{j-i}{2}}rkC_{i,j,s}(D).$$\end{definition}

\begin{theorem}

The graded Euler characteristic of the chain complex defined in this section is the Kauffman bracket of the link in the surface.

\end{theorem}

\begin{proof}

We will show that the graded Euler characteristic satisfies the three skein conditions of the Kauffman bracket.

\begin{enumerate}

\item

\begin{align*}
\chi(C_{*,*,s}(D\cup \bigcirc)) & = \sum_{i,j} A^j(-1)^{\frac{j-i}{2}}rkC_{i,j,s}(D \cup \bigcirc)  \\
& =\sum_{i,j} A^j(-1)^{\frac{j-i}{2}}(rkC_{i,j-2,s}(D)+rkC_{i,j+2,s}(D))\\
&   = \sum_{i,j} A^2A^{j-2}(-1)^{\frac{j-2-i+2}{2}}(rkC_{i,j-2,s}(D))  \\
& \quad + \sum_{i,j} A^{-2}A^{j+2}(-1)^{\frac{j+2-i-2}{2}}rkC_{i,j+2,s}(D)) \\
& = \sum_{i,j} A^2A^{j-2}(-1)^{\frac{j-2-i}{2}+1}(rkC_{i,j-2,s}(D)\\
& \quad +\sum_{i,j} A^{-2}A^{j+2}(-1)^{\frac{j+2-i}{2}-1}rkC_{i,j+2,s}(D)) \\                              
&  =  -A^2\sum_{i,j}A^{j-2}(-1)^{\frac{j-2-i}{2}}(rkC_{i,j-2,s}(D) \\
& \quad -A^{-2}\sum_{i,j} A^{j+2}(-1)^{\frac{j+2-i}{2}}rkC_{i,j+2,s}(D))\\
& =   (-A^2-A^{-2})\chi(C_{*,*,s}(D))                  
\end{align*}

\item
\begin{align*}
\chi(C_{*,*,s}(D)) & =  \sum_{i,j} A^j(-1)^{\frac{j-i}{2}}rkC_{i,j,s}(D) \\
&=  \sum_{i,j} A^j(-1)^{\frac{j-i}{2}}(rkC_{i-1,j-1,s}(D_+) + rkC_{i+1,j+1,s}(D_-)) \\
& =\sum_{i,j} A^j(-1)^{\frac{j-i}{2}}(rkC_{i-1,j-1,s}(D_+) \\
& \quad + \sum_{i,j}A^j(-1)^{\frac{j-i}{2}} rkC_{i+1,j+1,s}(D_-))    \\
& = \sum_{i,j} A*A^{j-1}(-1)^{\frac{j-1-(i-1)}{2}}(rkC_{i-1,j-1,s}(D_+) \\
& \quad + \sum_{i,j}A^{-1}A^{j+1}(-1)^{\frac{j+1-(i+1)}{2}} rkC_{i+1,j+1,s}(D_-))   \\ 
&     = A\sum_{i,j} A^{j-1}(-1)^{\frac{j-1-(i-1)}{2}}(rkC_{i-1,j-1,s}(D_+) \\
& \quad + A^{-1}\sum_{i,j}A^{j+1}(-1)^{\frac{j+1-(i+1)}{2}} rkC_{i+1,j+1,s}(D_-))   \\                       
&     = A\chi(C_{*,*,s}(D_+)) + A^{-1}\chi(C_{*,*,s}(D_-)) 
\end{align*}

\item

Consider $D$, a diagram of a link with no crossings.  The only possible non-trivial foams are ones with essential curves connected to essential curves in the bottom by vertical annuli.  Note the graded Euler characteristic is the representation of the link in the Kauffman bracket for surfaces.

\end{enumerate}

\end{proof}

\remark{The most significant advantage for using foams to generate the chain groups is a simpler proof of $d^2 =0$}

\section{Invariance of both homology theories}

We will prove that both theories are invariants of framed links in thickened surfaces simultaneously.  Note that all theorems and statements in this section can be applied to either theory.  The strategy to prove invariance will be the one employed by Bar-Natan in \cite{BN2} involving acyclic sub-complexes.  Thus we shall use the following theorem.

\begin{theorem}\cite{BN2}
\label{bn}
If $C'$ is acyclic, then $H(C/C') \cong  H(C)$, that is $C$ and C/C' produce the same homology.
\end{theorem}

\begin{remark}The disk without a dot acts like a unit as does the circle marked with a $v_-$ in \cite{BN2}.
\end{remark}

\subsection{Reidemeister I}

If the diagram has a positive twist we have the complex of complexes $C$ in Figure \ref{R1}.  $C$ has a subcomplex $C'$ where the disk never has a dot.  Note $C'$ is acyclic since the map has no kernel and is onto.

 \begin{figure}[h!]
\begin{center}

\subfigure[The complex of complexes, $C$, coming from a positive twist.]{
$C = \left[ \left[ \includegraphics[height = .2 in, trim = 0 70 0 0]{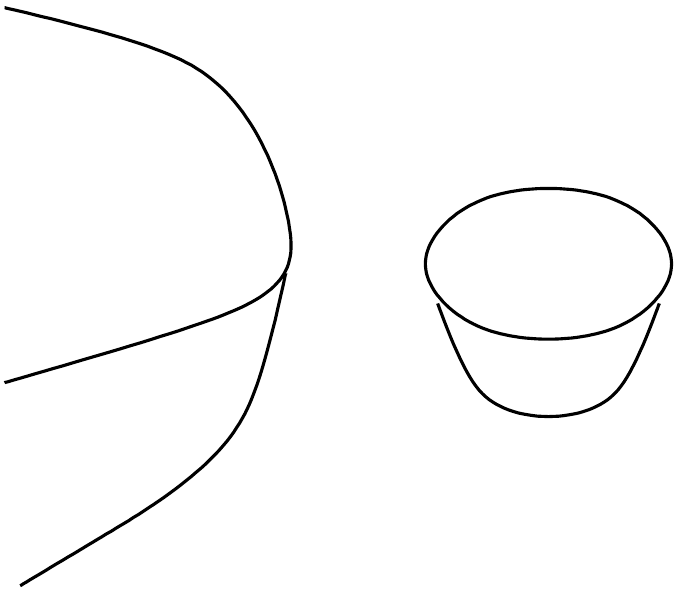} \right] \right] \rightarrow\left[ \left[ \includegraphics[height = .2 in, trim = 0 70 0 0]{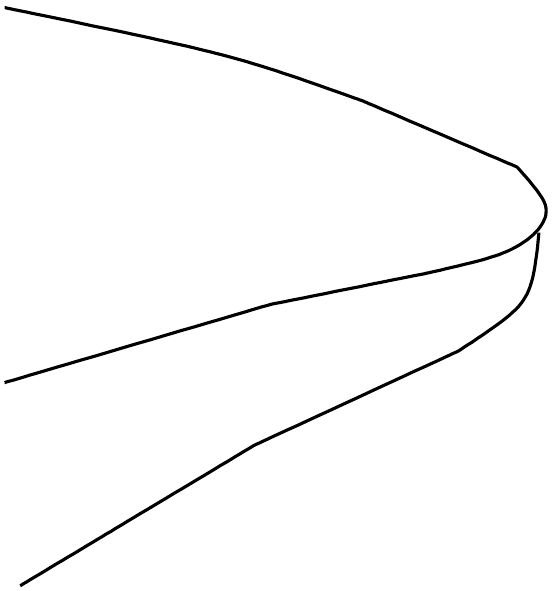} \right] \right]$
}
\hspace{.3 in}
\subfigure[The acyclic complex $C'$, where the disk never has a dot.]{
$ C' = \left[ \left[ \includegraphics[height = .2 in, , trim = 0 70 0 0]{fig1.pdf}  \right] \right]_{/\includegraphics[height = .1 in]{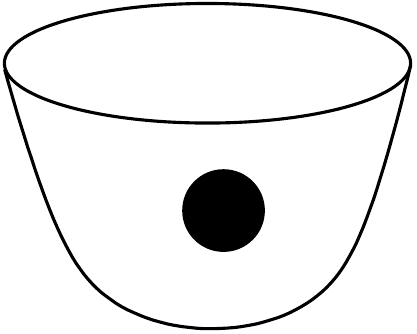}} \rightarrow$ $\left[ \left[ \includegraphics[height = .2 in, , trim = 0 70 0 0]{fig2.pdf}  \right] \right] $
}

\vspace{.2 in}

\subfigure[The complex $C/C'$, where the disk always has a dot.]{
\hspace{.3 in} $C/C' = \left[ \left[ \includegraphics[height = .2 in, , trim = 0 70 0 0]{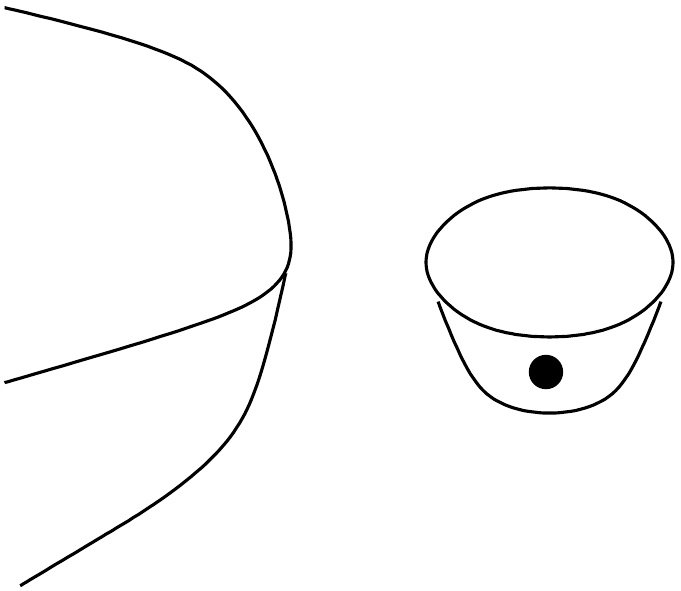}  \right] \right] \rightarrow  0$ \hspace{.3 in}
}
\hspace{.3 in}
\subfigure[The complex $D$, where the positive twist has been removed.]{
\hspace{.5 in} $D = \left[ \left[ \includegraphics[height = .2 in, , trim = 0 70 0 0]{fig2.pdf}  \right] \right]$ \hspace{.5 in}
}
\caption{The complexes associated with the Reidemeister one move.       \label{R1}}
\end{center}
\end{figure}

Note the complexes $C/C'$ and $D$ are identical apart from a shift in degrees.  Then by Theorem \ref{bn}, we have that $C$ and $D$ produce the same homology, up to degree.

To see how the degrees differ note that foams in $C/C'$ have an extra disk with a dot and an extra crossing that is smoothed positively.  Thus if $S$ is a foam in $D$ and $S'$ is the corresponding foam in $C/C'$ then $I(S) +1 = I(S')$,  $d(S) + 1= d(S')$ and $\chi(S) +1 = \chi(S')$, thus $J(S') = I(S') + 2(2d(S') - \chi(S')) = I(S) +1 +2(2(d(S) + 1) - (\chi(S) + 1)) = I(S) + 1 + 2(2d(S) + 2 -1 - \chi(S)) = I(S) + 2(2d(S) - \chi(S)) + 3$.  So $ I(S) + 1 = I(S')$ and $J(S) + 3 = J(S')$.   Thus we know exactly how a positive twist affects the homology.

\subsection{Reidemeister II}

Let $C$ be the complex in Figure \ref{R2} coming from one side of the Reidemeister two move.  Note $C$ has a subcomplex $C'$, depicted in Figure \ref{R2C'}, where the disk never has a dot.  $C'$ is acyclic since the only non-trivial map has no kernel and is onto. Now consider the complex $C/C'$ which is shown in Figure \ref{R2C/C'}.

\begin{figure}[h]

\begin{center}\begin{tabular}{cccccc}

$\left[ \left[ \includegraphics[height = .3 in, trim = 0 70 0 0]{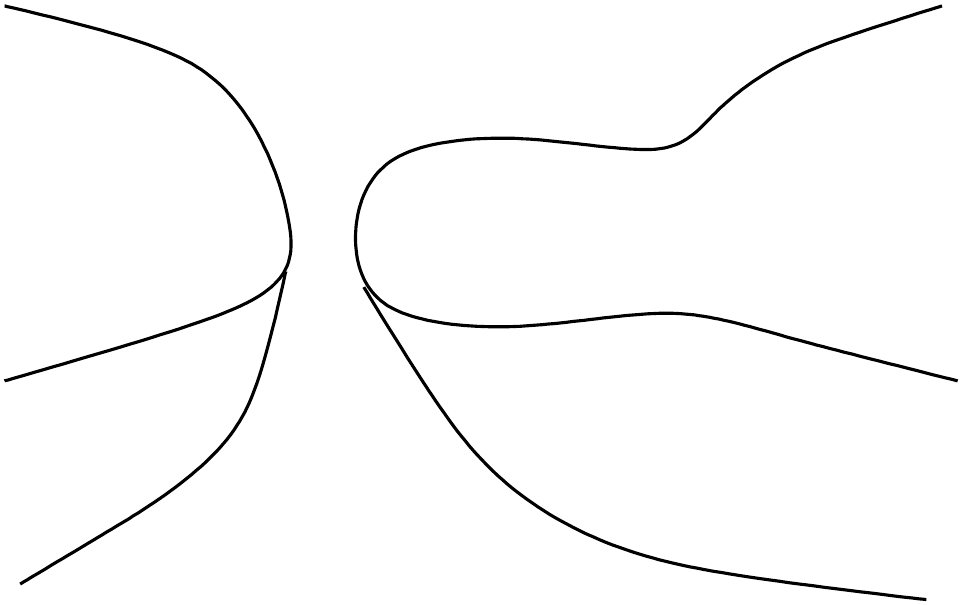} \right] \right]$& $\rightarrow$ & $

\left[\left[ \includegraphics[height = .3 in, trim = 0 70 0 0]{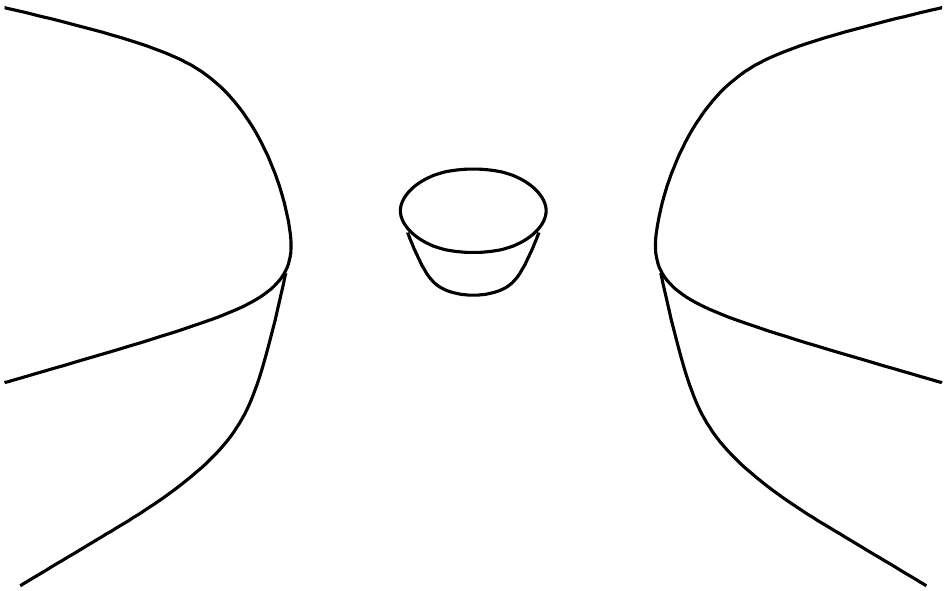} \right] \right]$ & & \\

$\downarrow$ & $C$ & $\downarrow$ \\

$\left[ \left[ \includegraphics[width = .8 in, trim = 0 50 0 0]{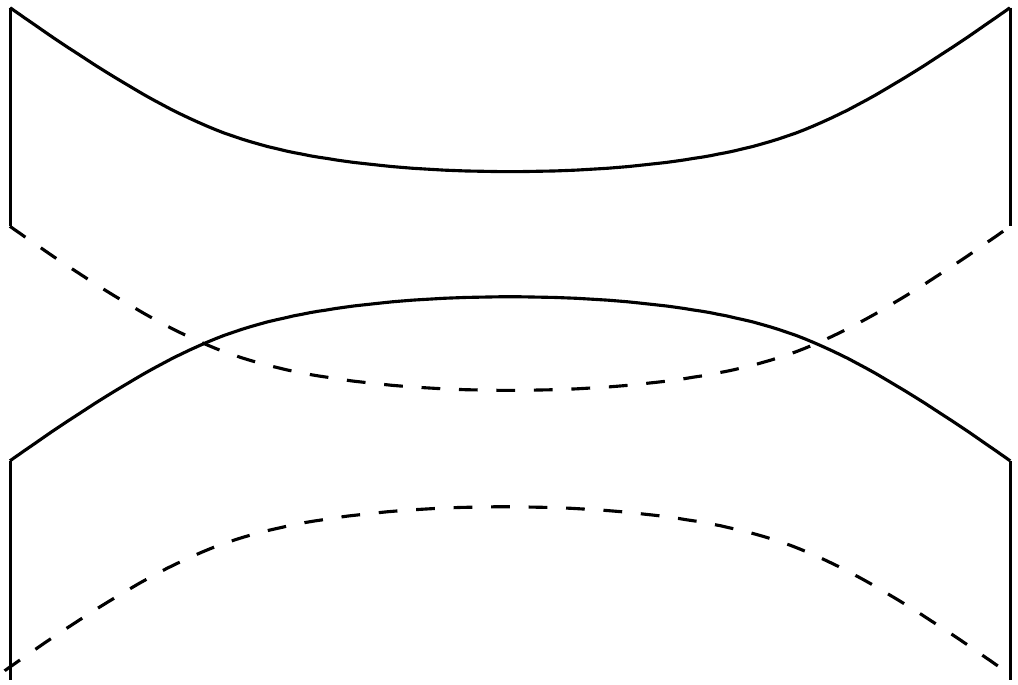} \right] \right]$ & $\rightarrow$& $

\left[ \left[ \includegraphics[height = .3 in, trim = 0 70 0 0]{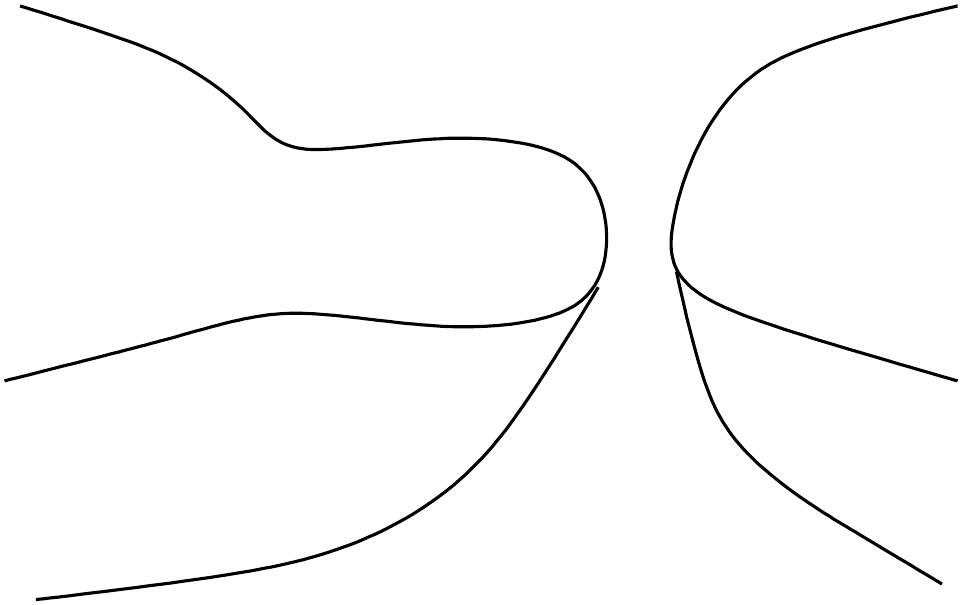} \right] \right]$ \\

\end{tabular}\end{center}

\caption{The complex coming from one side of the Reidemeister two move. \label{R2}}
\end{figure}

\begin{figure}[h]
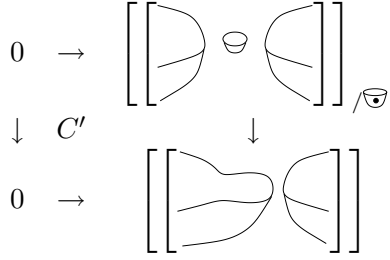


\begin{center}\begin{tabular}{cccccc}

0& $\rightarrow$ & 

$\left[ \left[ \includegraphics[height = .3 in,  trim = 0 70 0 0]{fig4.pdf} \right] \right] _{/\includegraphics[height = .1 in]{diskdot.pdf}}$& \\

$\downarrow$ & $C'$ & $\downarrow$ \\

0 & $\rightarrow$& $\left[ \left[ \includegraphics[height = .3 in, trim = 0 70 0 0]{fig5.pdf} \right] \right]$ \\

&&&&

\end{tabular}

\end{center}
\caption{The acyclic subcomplex $C'$. \label{R2C'}}
\end{figure}

\medskip

\begin{figure}[h]

\begin{center}\begin{tabular}{cccccc}

$\left[ \left[ \includegraphics[height = .3 in,  trim = 0 70 0 0]{fig6.pdf} \right] \right]$& $\rightarrow$ \hspace{-.2 in}\raisebox{.1 in}{$\Delta$} & 

$\left[ \left[ \includegraphics[height = .3 in,  trim = 0 70 0 0]{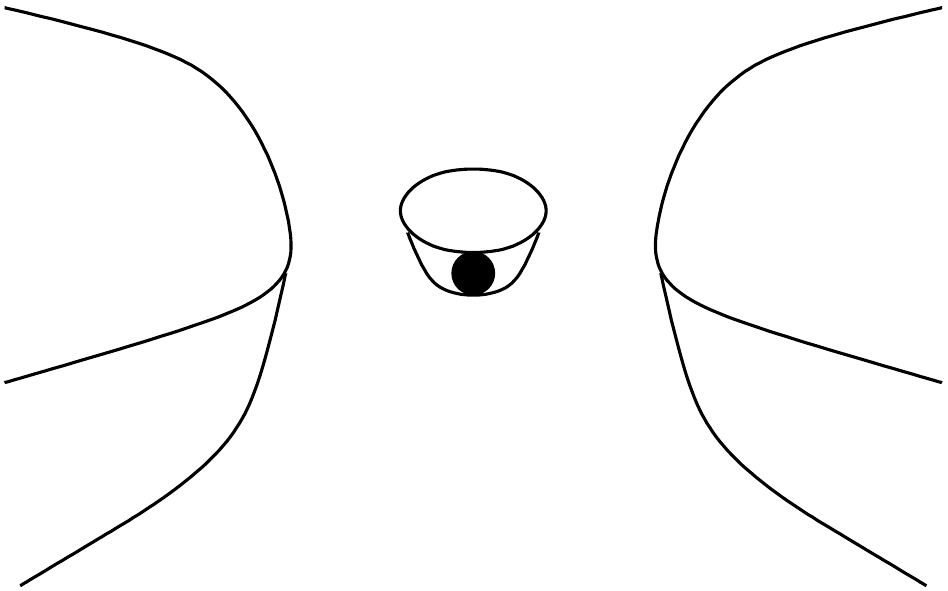} \right] \right]$ & & \\

$\downarrow m $ & $C/C'$ & $\downarrow$ \\

 $\left[ \left[ \includegraphics[height = .3 in,,  trim = 0 70 0 0]{fig91.pdf} \right] \right]$ & $\rightarrow$& 0 \\

&&&&

\end{tabular}

\end{center}
\caption{The complex $C/C'$.    \label{R2C/C'}}
\end{figure}

\medskip

Now we define $\tau: \left[ \left[ \includegraphics[height = .3 in,,  trim = 0 70 0 0]{fig8.pdf} \right] \right] \rightarrow \left[ \left[ \includegraphics[height = .3 in,,  trim = 0 70 0 0]{fig91.pdf} \right] \right]$ on $C/C'$, by 

\begin{center}$\tau\left(\includegraphics[height = .3 in,,  trim = 0 70 0 0]{fig8.pdf}\right)=$ \includegraphics[height = .2 in,  trim = 0 70 0 0]{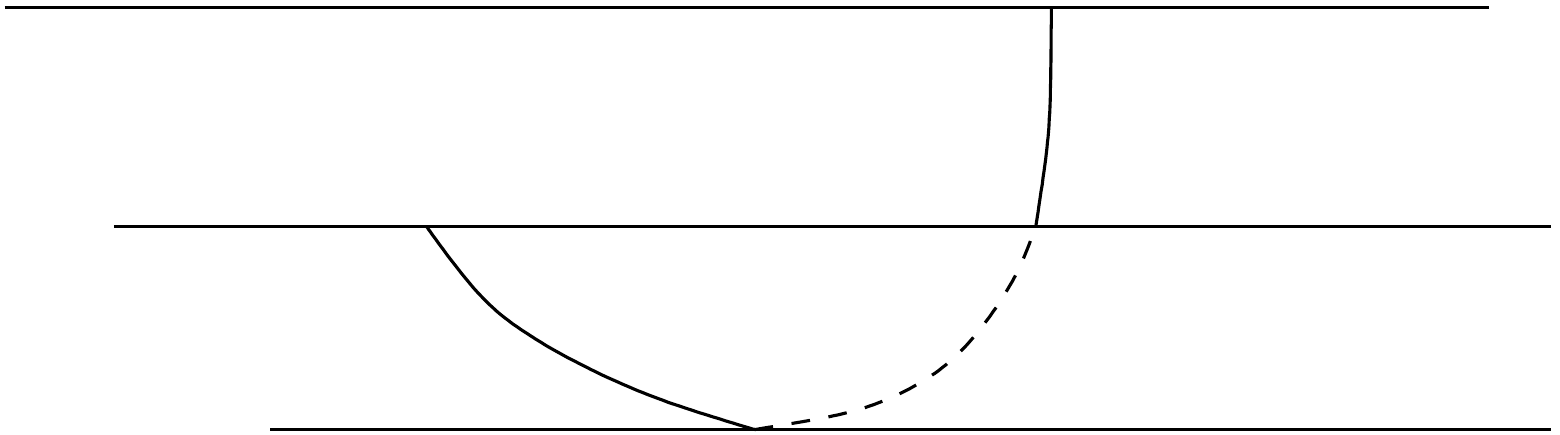} \end{center}

with everything outside the picture being the same. Note $\tau = m \circ \Delta^{-1}$.

\bigskip

Note $C/C'$ has a subcomplex 

\begin{center}$C'':$ 
$\left[ \left[ \includegraphics[height = .3 in,  trim = 0 70 0 0]{fig6.pdf} \right] \right] \rightarrow \left[ \left[\includegraphics[height = .3 in,  trim = 0 70 0 0]{fig8.pdf} \right] \right] \oplus \tau \left(\left[ \left[ \includegraphics[height = .3 in,  trim = 0 70 0 0]{fig8.pdf} \right] \right] \right) \rightarrow 0 $ \end{center}

\medskip

where the second term in the complex is elements of the form $(\beta,\tau(\beta))$.

\bigskip

$C''$ is acyclic since the non-zero map is onto the first factor and by construction of $\tau$ (since $\Delta$ is a bijection) is onto the second factor as well.

\bigskip

Modding out by $C''$ results in the complex with the disk being identified with elements in the complex without the disk, so we have

\begin{center}$(C/C')/C''$  is 0 $\rightarrow \left[ \left[ \includegraphics[height = .2 in,  trim = 0 70 0 0]{fig91.pdf} \right] \right] \rightarrow 0$,  \end{center}

which is the same complex that appears if a Reidemeister II move was performed on the original diagram.  By Theorem \ref{bn}, the homology of the original complex is equivalent to the homology produced by $(C/C')/C''$ and thus it is unchanged by a Reidemeister II move.

\subsection{Reidemeister III}

Consider the two cubes that come from the Reidemeister three move, $C_A$ and $C_B$, as shown in figure \ref{R3}.
 
\medskip

\begin{figure}[h!]
\scalebox{.85}{
\begin{tabular}{cc}
\includegraphics[height = 3 in]{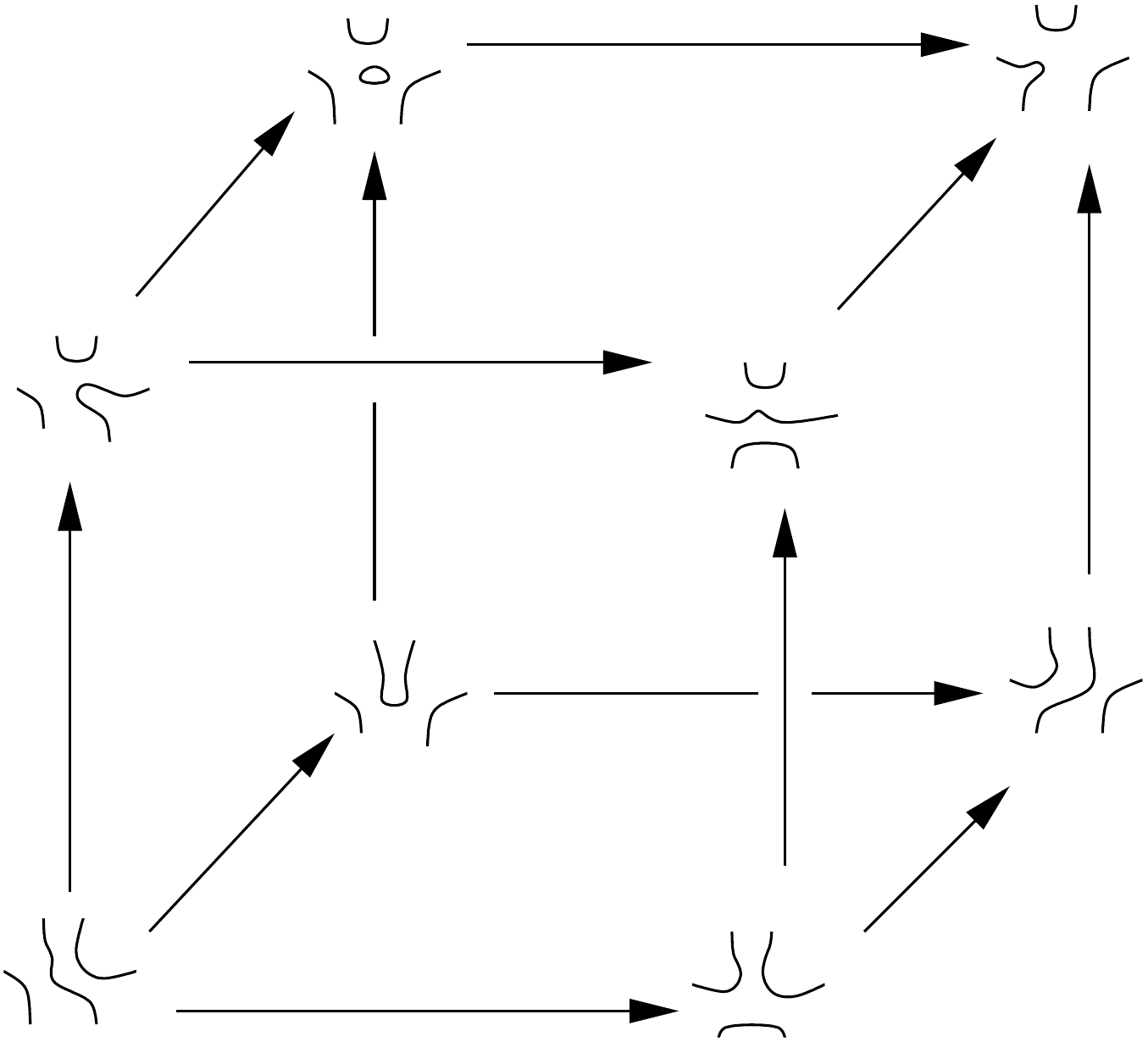}  &   \includegraphics[height = 3 in]{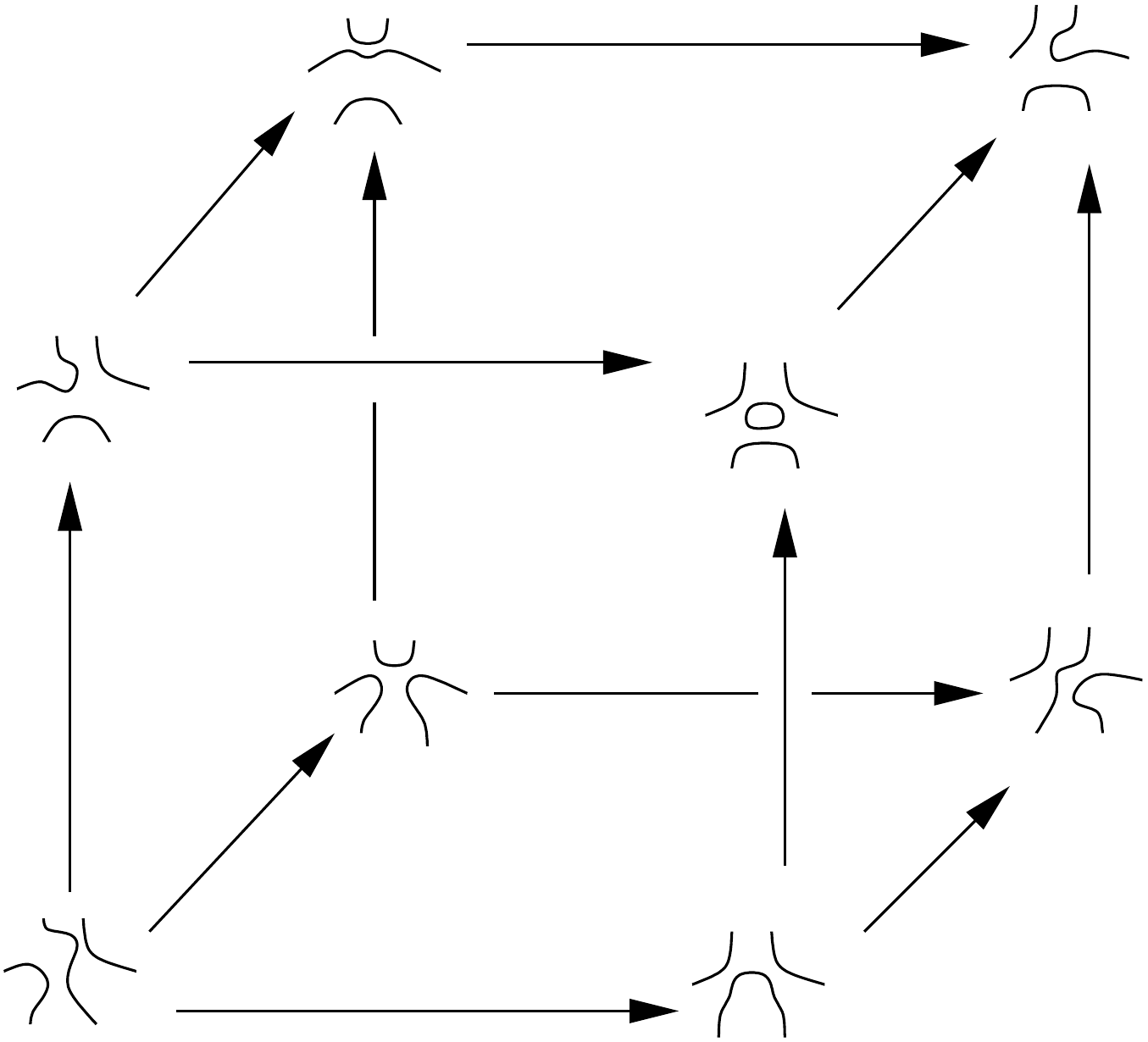} \\
$C_A$  &  $C_B$  \\
\end{tabular}
}
\caption{The complexes coming from the Reidemeister three move. \label{R3}}
\end{figure}

Notice the two subcomplexes, $C_A'$ and $C_B'$ in Figure \ref{R3C'}, where the disk never has a dot, are acyclic as in the Reidemeister two case.  Now consider the complexes $C_A/C_A'$ and  $C_B/C_B'$ in Figure \ref{R3C/C'}, where the disk always has a dot.

\begin{figure}[h!]
\scalebox{.85}{
\begin{tabular}{cc}

\includegraphics[height = 3 in]{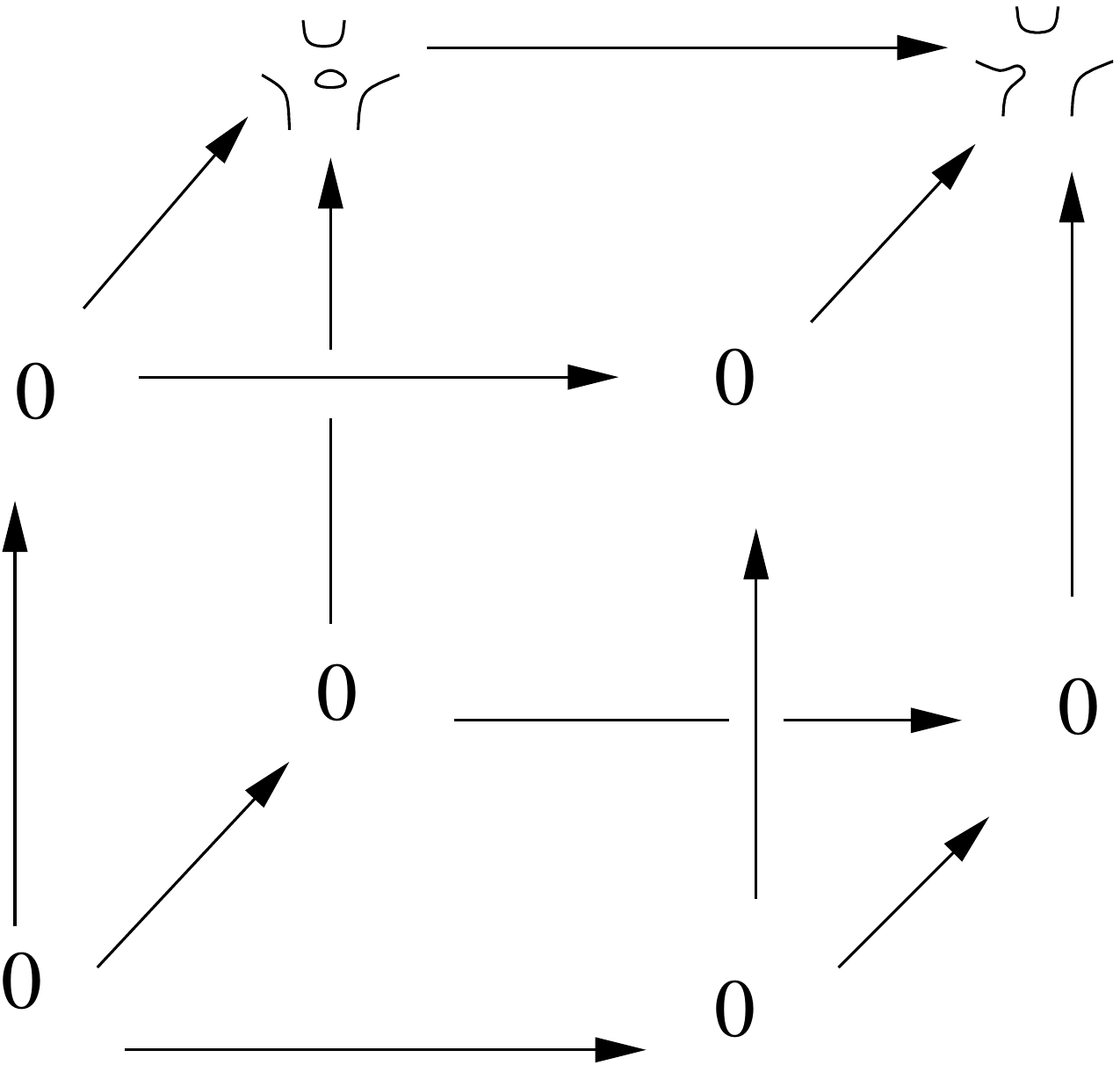}\put(-150,190){/\includegraphics[height = .1 in]{diskdot.pdf}}   &  \includegraphics[height = 3 in]{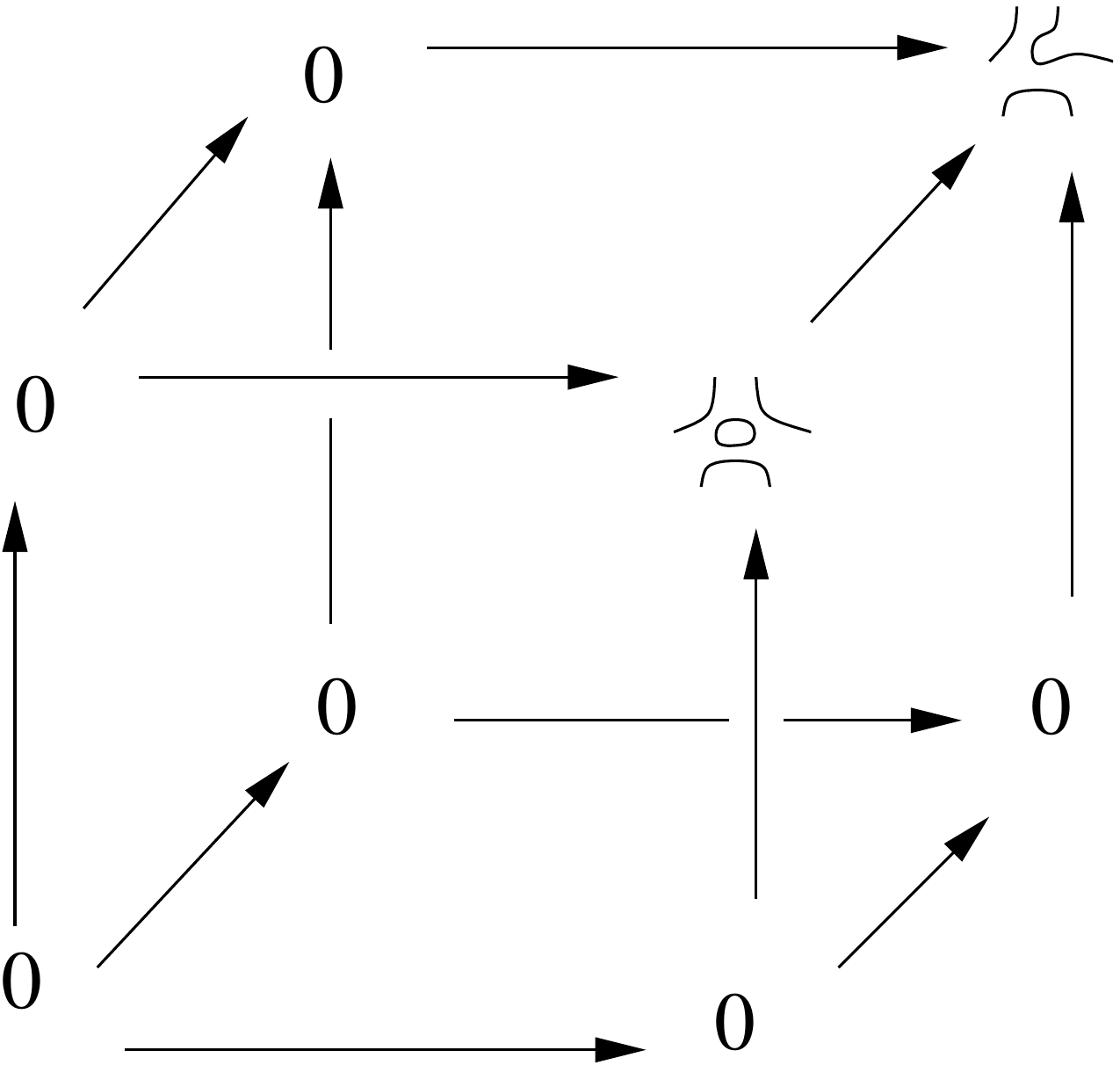}\put(-60,120){/\includegraphics[height = .1 in]{diskdot.pdf}} \\
$C_A'$  &  $C_B'$  \\
\end{tabular}
}
\caption{The acyclic subcomplexes $C_A'$ and $C_B'$.      \label{R3C'}}
\end{figure}

\begin{figure}[h!]
\scalebox{.85}{
\begin{tabular}{ccc}
\includegraphics[height = 3 in]{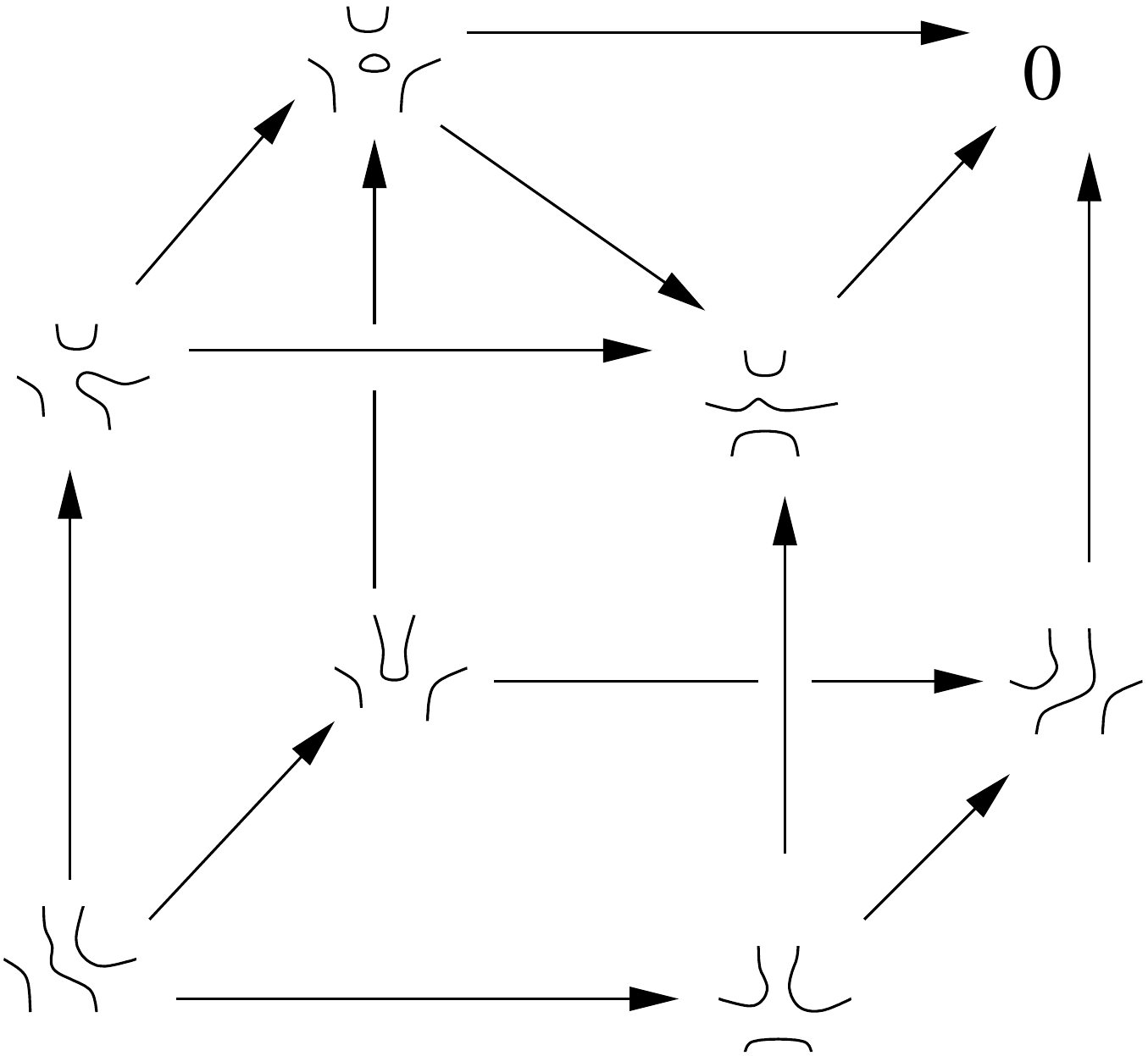}\put(-135,170){$\tau$}   &  \includegraphics[height = 3 in]{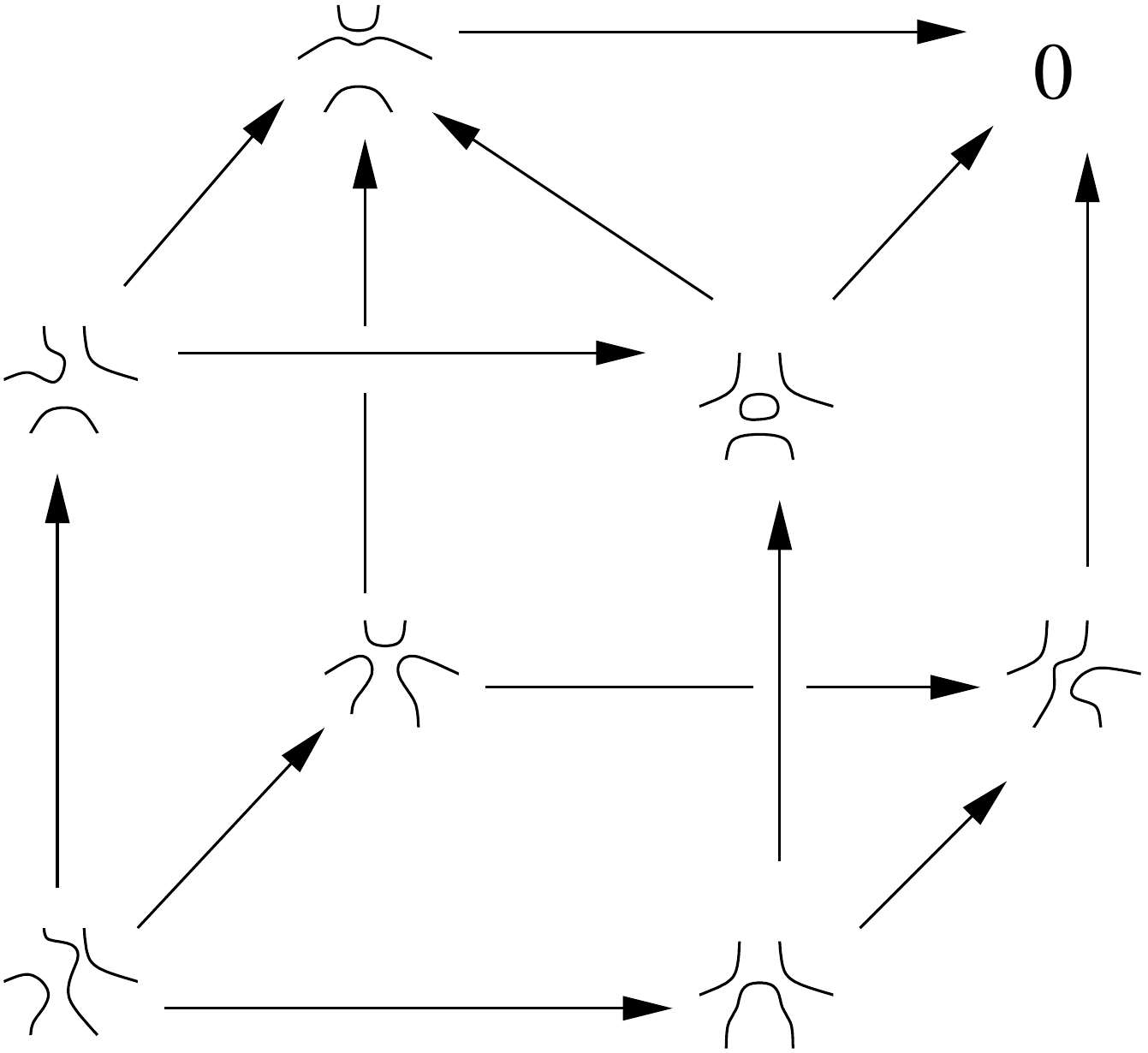}\put(-135,170){$\tau$} \\

$C_A/C_A'$ &  $C_B/C_B'$

\end{tabular}
}
\caption{The complexes $C_A/C_A'$ and  $C_B/C_B'$.  \label{R3C/C'} }
\end{figure}

\medskip
Let $\tau$ be the same map as defined in the Reidemeister two case.  Then we can consider the acyclic subcomplexes of the previous cubes, $C_A''\subset C_A/C_A' $  and $C_B''\subset C_B/C_B'$ as shown in Figure \ref{R3C''}.  The complexes are acylic because $\Delta$ is an isomophism in this complex so the only non-zero map is one-to-one and onto.

\medskip

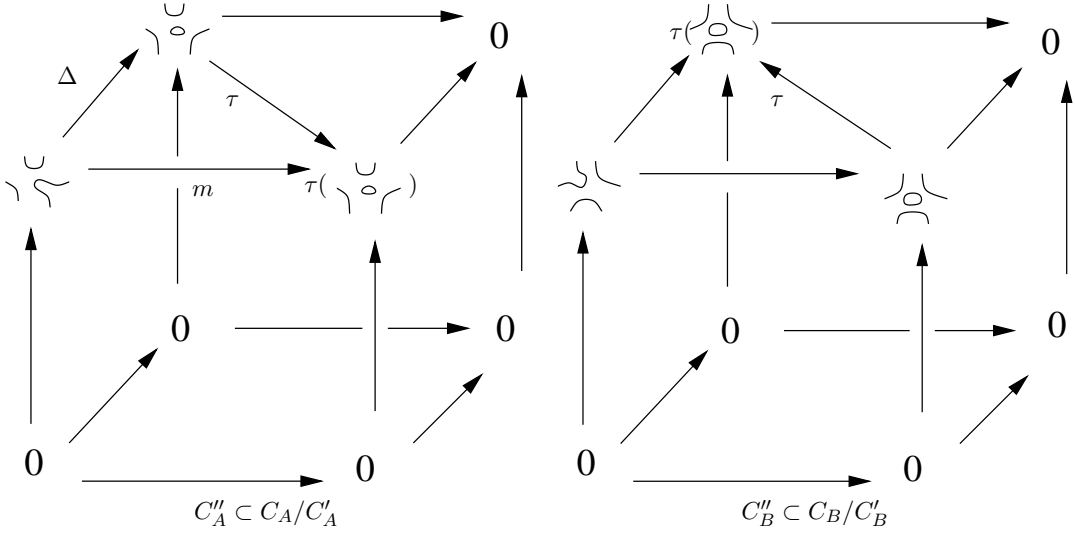
\begin{figure}[h!]
\scalebox{.85}{
\begin{tabular}{ccc}
\includegraphics[height = 3 in]{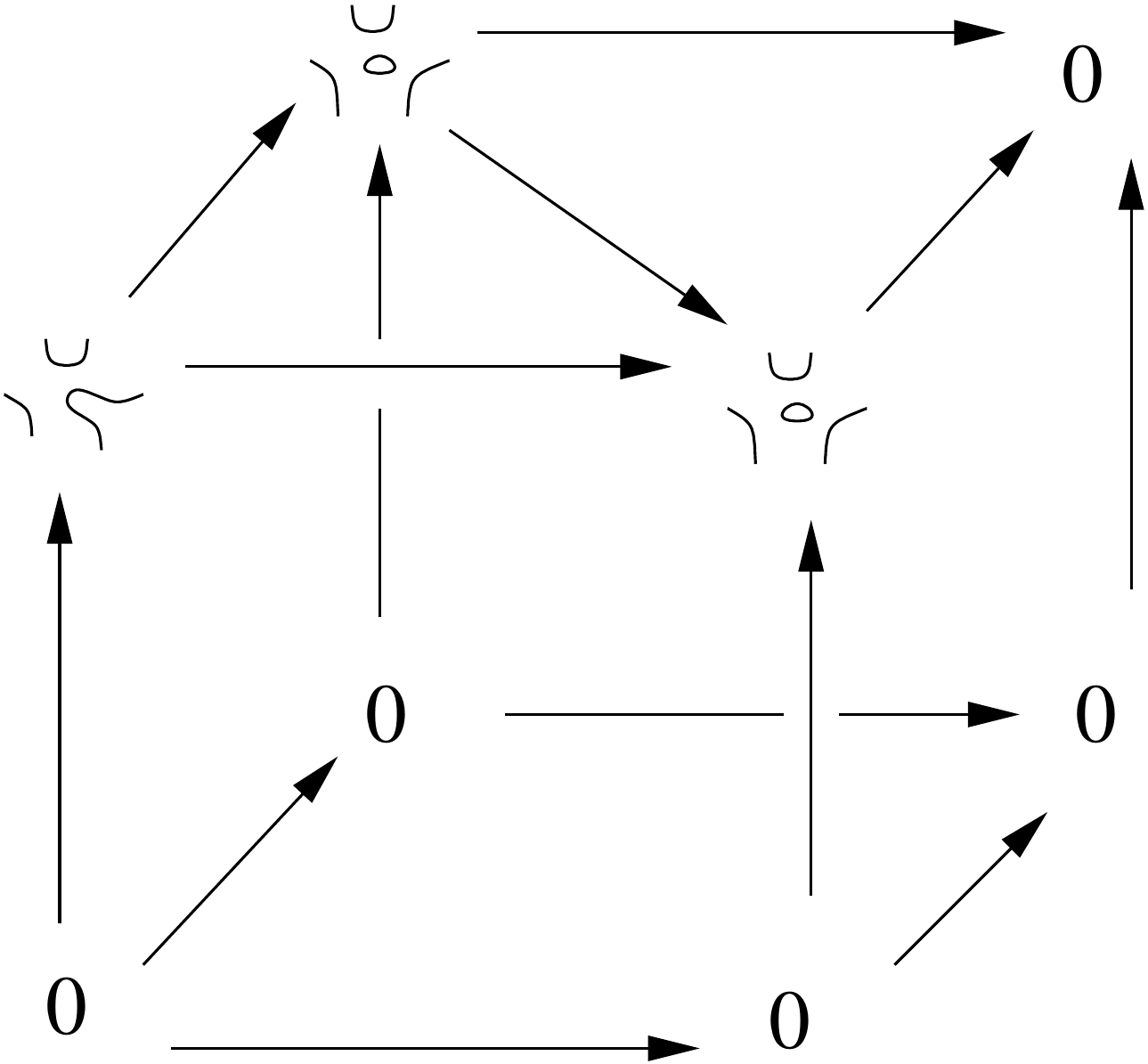}\put(-210,180){$\Delta$}\put(-150,130){$m$}\put(-135,170){$\tau$}\put(-100,132){$\tau($\hspace{.5in}$)$}    &  \includegraphics[height = 3 in]{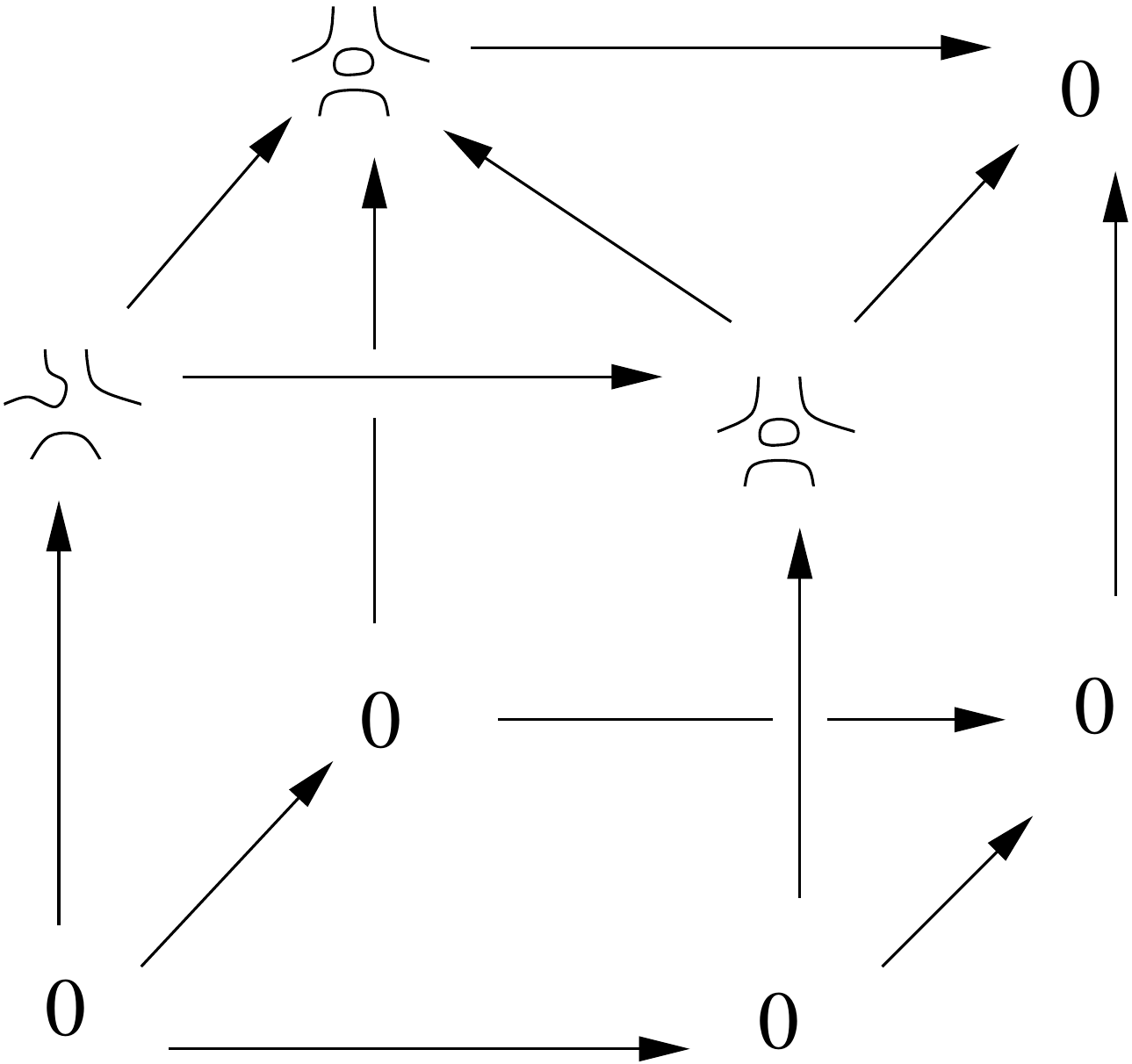}\put(-135,170){$\tau$}\put(-180,200){$\tau($\hspace{.36in}$)$} \\

$C_A''\subset C_A/C_A' $  & $C_B''\subset C_B/C_B'$\\
\end{tabular}
}
\caption{ The complexes $C_A''$ and $C_B''$.      \label{R3C''}}
\end{figure}

\medskip

Finally we arrive at $ (C_A/C_A')/C_A'' $ and $ (C_B/C_B')/C_B''$ in Figure \ref{R3final}.  Note that when we mod out by $C_A''$ and $C_B''$ we essentially set $\beta = \tau (\beta)$.

\begin{figure}[h!]
\scalebox{.85}{
\begin{tabular}{ccc}

\includegraphics[height = 3 in]{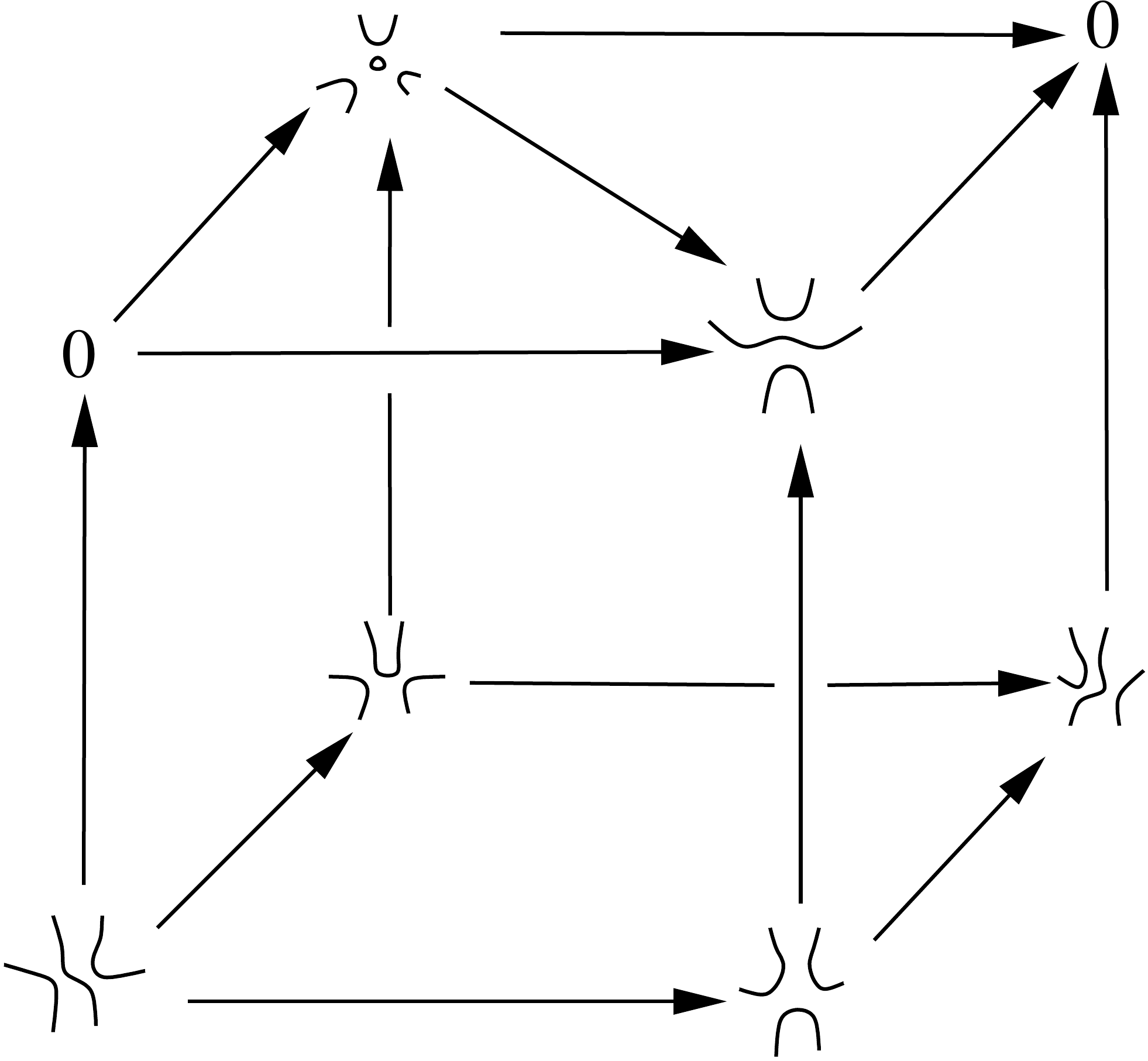}\put(-250,80){$d_{0,0,*}^A$}\put(-180,40){$d_{*,0,0}^A$}\put(-150,20){$d_{0,*,0}^A$}\put(-150,150){$d_{0,*,1}^A$}\put(-65,105){$d_{0,1,*}^A$} \put(-25,45){$d_{*,1,0}^A$}\put(0,150){$d_{1,1,*}^A$}\put(-100,220){$d_{1,*,1}^A$}\put(-210,180){$d_{*,0,1}^A$}\put(-40,165){$d_{*,1,1}^A$}\put(-180,105){$d_{1,0,*}^A$}\put(-110,80){$d_{1,*,0}^A$}\put(-135,175){$\tau^A$} &  & \includegraphics[height = 3 in]{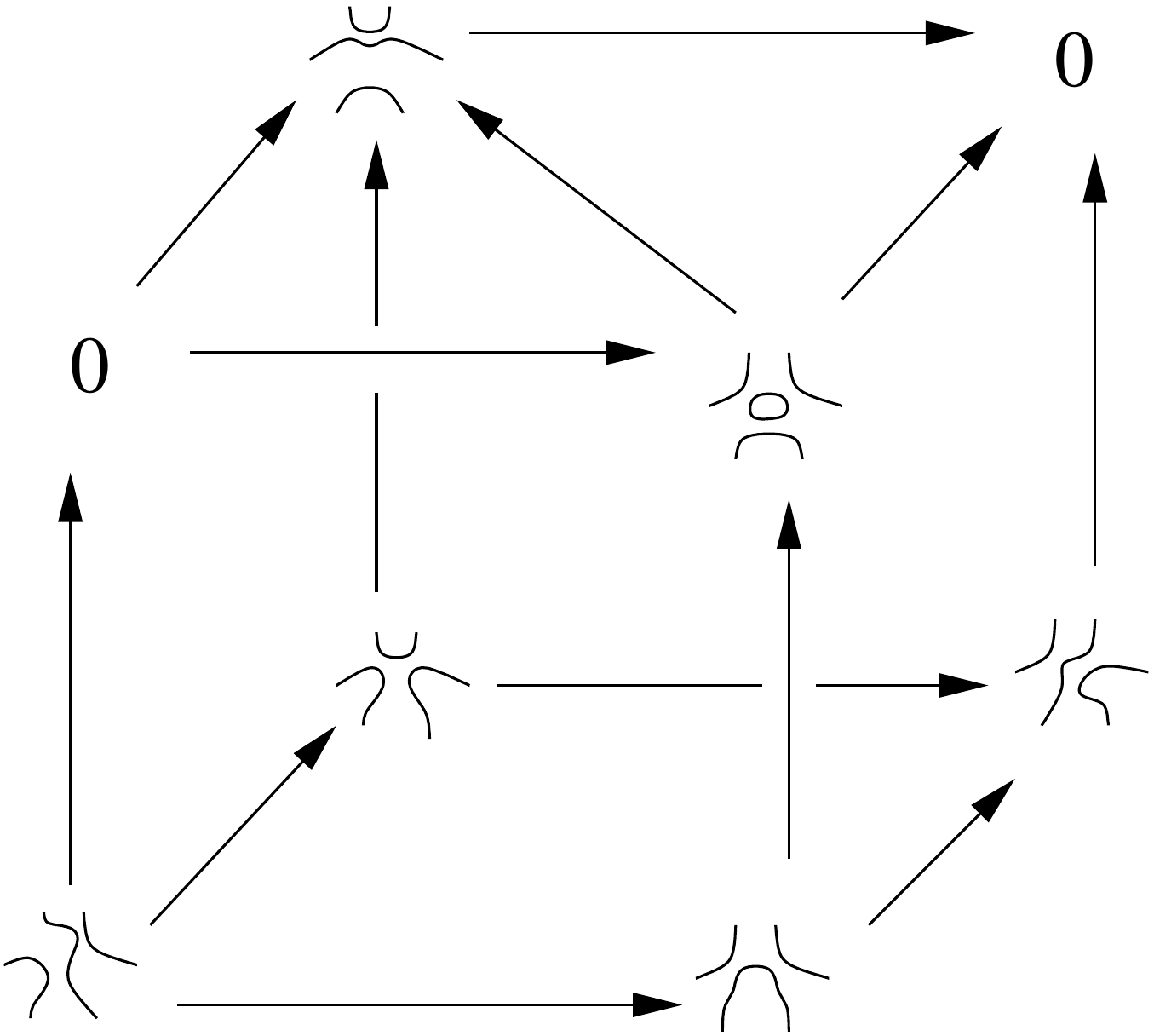}\put(-250,80){$d_{0,0,*}^B$}\put(-180,40){$d_{*,0,0}^B$}\put(-150,20){$d_{0,*,0}^B$}\put(-150,150){$d_{0,*,1}^B$}\put(-65,105){$d_{0,1,*}^B$} \put(-25,45){$d_{*,1,0}^B$}\put(0,150){$d_{1,1,*}^B$}\put(-100,220){$d_{1,*,1}^B$}\put(-215,180){$d_{*,0,1}^B$}\put(-40,165){$d_{*,1,1}^B$}\put(-190,105){$d_{1,0,*}^B$}\put(-110,80){$d_{1,*,0}^B$}\put(-115,175){$\tau^B$} \\

$ (C_A/C_A')/C_A'' $  & &  $ (C_B/C_B')/C_B''$ \\

\end{tabular}
}
\caption{The complexes $ (C_A/C_A')/C_A'' $ and  $ (C_B/C_B')/C_B''$.        \label{R3final}}
\end{figure}
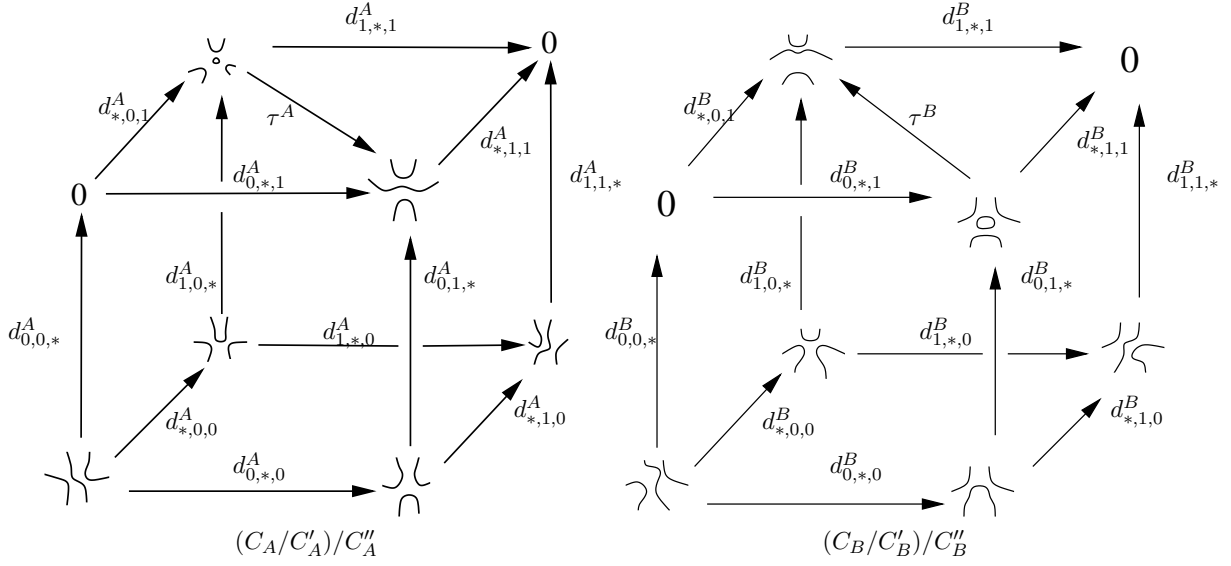

\medskip

We will define $\Upsilon$ as in [BN2],  $\Upsilon$ sends the bottom layer of $ (C_A/C_A')/C_A'' $ to the bottom layer of $ (C_B/C_B')/C_B'' $.  $\Upsilon$ does almost the same to the top of the cubes, except that the top layers are reversed so that the foams with the closed component locally are mapped to each other.  We must now verify that $d \circ \Upsilon = \Upsilon \circ d$ so that $\Upsilon$ is a chain map.  Since the tops and bottoms are immediately isomorphic, the only thing left to show is that $\tau^A \circ d_{1,0,*}^A=d_{1,0,*}^B$ and $\tau^B \circ d_{0,1,*}^B=d_{0,1,*}^A$.  

\medskip

Note 

\begin{equation*}\tau^A \circ d_{1,0,*}^A \left( \includegraphics[height = .3 in, trim = 0 70 0 0]{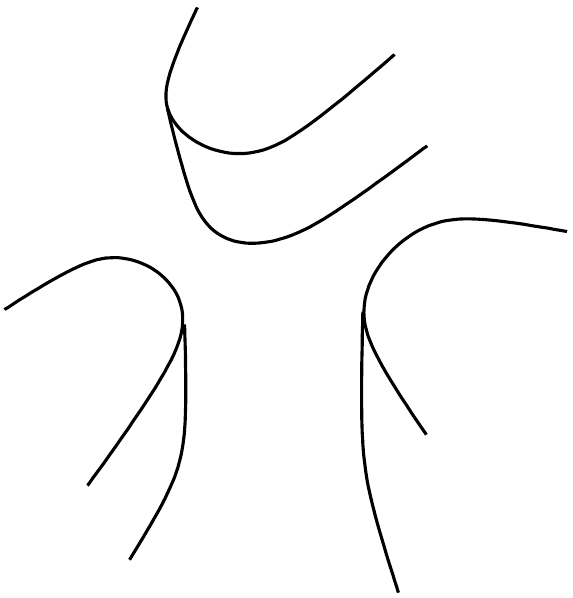} \right) = \tau^A \left( \includegraphics[height = .3 in, trim = 0 70 0 0]{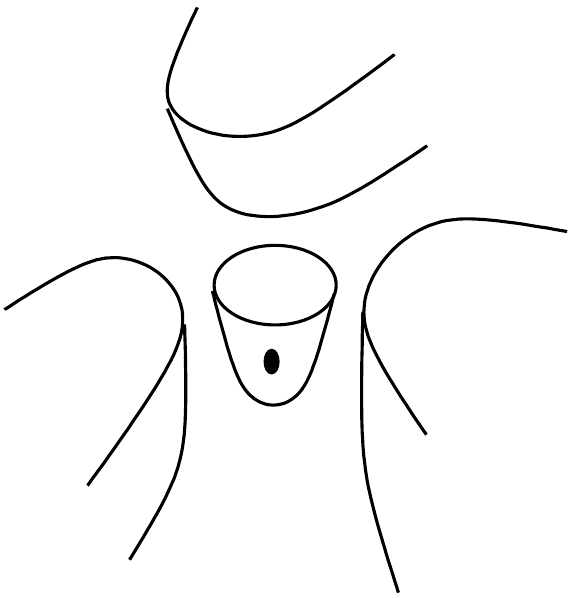} \right) = \includegraphics[height = .2 in, trim = 0 70 0 0]{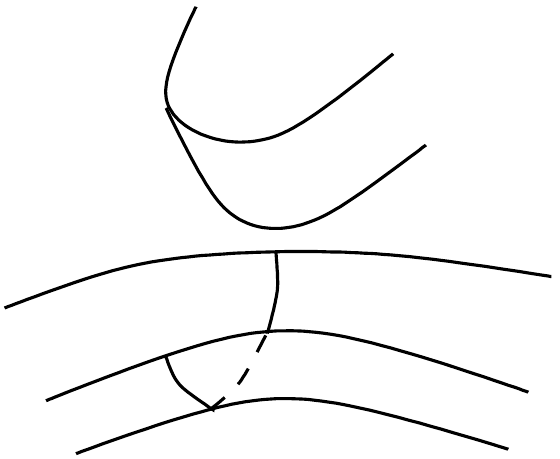}\end{equation*}

and,

\begin{equation*}d_{1,0,*}^B \left( \includegraphics[height = .3 in, trim = 0 70 0 0]{taumap1.pdf} \right) = \includegraphics[height = .2 in, trim = 0 70 0 0]{taumap3.pdf},\end{equation*}

so we have $\tau^A \circ d_{1,0,*}^A=d_{1,0,*}^B$.  Showing that  $\tau^B \circ d_{0,1,*}^B=d_{0,1,*}^A$ is done in the exact same manner, thus we have that $\Upsilon$ is a chain map.

\begin{center}
\begin{tabular}{cccccc}

Bottom$_A$ & $\rightarrow$ & Top$_A$ & & \\

$\downarrow \cong \uparrow$ & $\circlearrowright$ & $\downarrow \cong \uparrow$ \\

Bottom$_B$ & $\rightarrow$& Top$_B$ \\

\end{tabular}\put(-137,0){$\Upsilon$}\put(-30,0){$\Upsilon$}
\end{center}

Since the diagram commutes and both maps are isomorphisms the chain complexes are isomorphic.  These chain complexes produce the same homology as $C_A$ and $C_B$ respectively since they are quotients of $C_A$ and $C_B$ by acyclic complexes.  Thus $C_A$ and $C_B$ produce the same homology and the theory is invariant under the third Reidemeister move.  This proves

\begin{theorem}

Both homology theories that have been defined are invariants of framed links and $H_{i,j,s}\left(\includegraphics[height=.3in, trim=0 70 0 0]{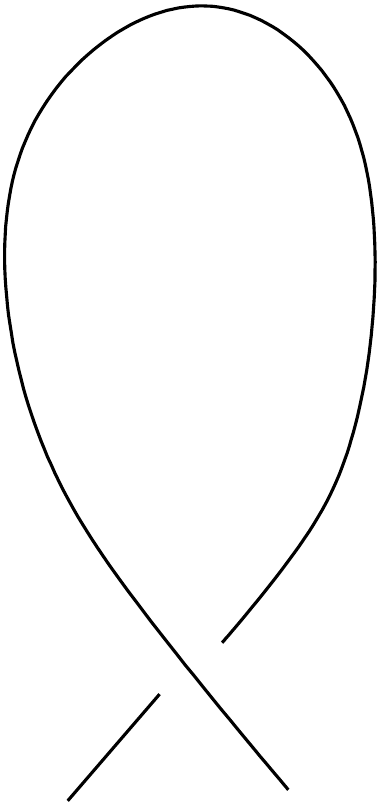}\right) = H_{i-1,j-3,s}\left(\includegraphics[height=.3in, trim=0 55 0 0]{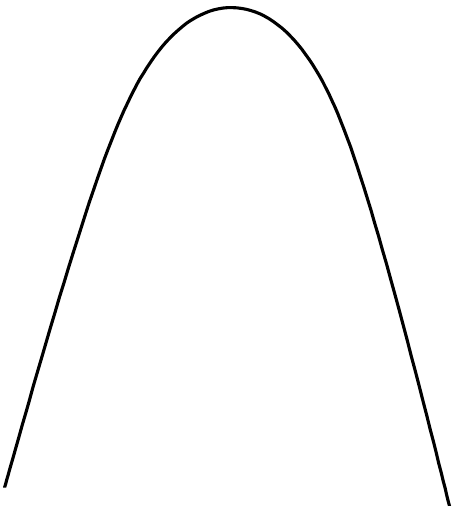}\right)$.

\end{theorem}

Department of Mathematics, University of Iowa, Iowa City, IA 52245, USA

E-mail: jboerner@math.uiowa.edu

\end{document}